\documentclass[11pt,a4paper]{amsart}
\usepackage{a4wide}

\usepackage[utf8]{inputenc}
\usepackage[english]{babel}

\usepackage{amsfonts,amssymb,amsmath}

\usepackage{xcolor}
\usepackage[colorlinks,
    linkcolor={red!50!black},
    citecolor={blue!50!black},
    urlcolor={blue!80!black}]{hyperref}

\usepackage{tikz}
\usepackage[all]{xy}
\usepackage{soul}
\usetikzlibrary{calc, matrix, arrows, cd}

\newtheorem{innercustomthm}{Theorem}
\newenvironment{customthm}[1]
  {\renewcommand\theinnercustomthm{#1}\innercustomthm}
  {\endinnercustomthm}

\newtheorem{theorem}{Theorem}[section]

\newtheorem{lemma}[theorem]{Lemma}
\newtheorem{proposition}[theorem]{Proposition}

\theoremstyle{definition}
\newtheorem{definition}[theorem]{Definition}

\newtheorem{example}[theorem]{Example}

\newtheorem{remark}[theorem]{Remark}

\newtheorem*{claim}{Claim}

\newcommand{\Z}{\mathbb{Z}}

\newcommand{\R}{\mathbb{R}}

\newcommand{\id}{\operatorname{id}}

\newcommand{\im}{\operatorname{im}}

\newcommand{\Grcc}{\overline{\mathcal{B}}_{(c,c)}}
\newcommand{\Grccp}{\overline{\mathcal{B}}_{(c,c')}}

\newcommand{\Grcpcp}{\overline{\mathcal{B}}_{(c',c')}}

\begin{document}

\title{Conway's potential function via the Gassner representation}
\author{Anthony Conway}
\address{Universit\'e de Gen\`eve, Section de math\'ematiques, 2-4 rue du Li\`evre, 1211 Gen\`eve 4, Switzerland}
\email{anthony.conway@unige.ch}
\author{Solenn Estier}
\address{Universit\'e de Gen\`eve, Section de math\'ematiques, 2-4 rue du Li\`evre, 1211 Gen\`eve 4, Switzerland}
\email{solenn.estier@etu.unige.ch}

\begin{abstract}
We show how Conway's multivariable potential function can be constructed using braids and the reduced Gassner representation. The resulting formula is a multivariable generalization of a construction, due to Kassel-Turaev, of the Alexander-Conway polynomial in terms of the Burau representation. Apart from providing an efficient method of computing the potential function, our result also removes the sign ambiguity in the current formulas which relate the multivariable Alexander polynomial to the reduced Gassner representation. We also relate the distinct definitions of this representation which have appeared in the literature.
\end{abstract}
\maketitle

\section{Introduction}

The one variable Alexander polynomial of an oriented link $L$ is a Laurent polynomial $\Delta_L(t) \in \Z[t^{\pm 1}]$ which is defined up to multiplication by $\pm t^{k}$ with $k \in \Z$. Despite this indeterminacy, $\Delta_L(t)$ has proved invaluable in low dimensional topology and can be understood in a wealth of different ways. For instance, $\Delta_L(t)$ can be constructed using Seifert surfaces~\cite{Seifert}, the reduced Burau representation~\cite{Burau}, Fox calculus~\cite{FoxFree2}, Reidemeister torsion~\cite{MilnorDuality}, quantum invariants~\cite{KauffmanSaleur, DeguchiAkutsu} and Heegaard-Floer homology~\cite{OzsvathSzaboHolomorphic}.

These considerations extend to the multivariable case. Indeed, the \emph{multivariable Alexander polynomial} of an $n$-component ordered link $L$ is a Laurent polynomial $\Delta_L(t_1,\ldots,t_n) \in \Z[t_1^{\pm 1},\ldots,t_n^{\pm 1}]$ which is defined up to multiplication by powers of~$\pm t_i$. Analogously to the one variable case, $\Delta_L(t_1,\ldots,t_n)$ can be constructed using generalized Seifert surfaces~\cite{CimasoniPotential}, Fox calculus~\cite{FoxFree2}, the reduced Gassner representation~\cite{Birman}, Reidemeister torsion~\cite{TuraevReidemeister}, quantum invariants~\cite{Murakami} and Heegaard-Floer homology~\cite{OzsvathSzabo}.

Regardless of the number of variables, the Alexander polynomial is palindromic, i.e. it satisfies $\Delta_L(t_1^{-1},\ldots,t_n^{-1}) \stackrel{.}{=} \Delta_L(t_1,\ldots,t_n)$, where $\stackrel{.}{=}$ denotes equality up to multiplication by a unit of $\Z[t_1^{\pm 1},\ldots,t_n^{\pm 1}]$. Consequently, the difficulty in removing the indeterminacy lies in fixing a signed representative in~$\Z[t_1^{\pm 1/2},\ldots,t_n^{\pm 1/2}]$. In 1970, J. Conway~\cite{ConwayJohn} suggested such a representative (later called the~\emph{Conway potential function}) of the multivariable Alexander polynomial. Namely, the potential function of an $n$-component ordered link~$L$ is a rational function $\nabla_L(t_1,\ldots,t_n)$ which satisfies 
$$  \nabla_L(t_1,\ldots,t_n)=
\begin{cases}
\frac{1}{t_1-t_1^{-1}}\Delta_L(t_1^2) & \mbox{if }  n=1, \\
\Delta_L(t_1^2,\ldots,t_n^2)   & \mbox{if } n>1.\\ 
 \end{cases} $$ 
In the one variable case, J. Conway further defined the \emph{reduced polynomial} $D_L(t) \in \Z[t^{\pm 1}]$ of a link by setting $D_L(t)=(t-t^{-1})\nabla_L(t)$. The existence of this Laurent polynomial (which is now called the \emph{Alexander-Conway polynomial}) was first proved by Kauffman~\cite{KauffmanConway} using Seifert surfaces. Subsequent constructions involve quantum invariants~\cite{KauffmanSaleur}, Heegaard-Floer homology~\cite{OzsvathSzaboHolomorphic} and the Burau representation of the braid group~\cite[Section 3.4]{TuraevKassel}.

In the multivariable case, the existence of the potential function was first proved by Hartley~\cite{Hartley} using Fox calculus. Furthermore, $\nabla_L(t_1,\ldots,t_n)$ can currently be expressed by sign-refining the aforementioned constructions of~$\Delta_L(t_1,\ldots,t_n)$~\cite{CimasoniPotential, TuraevReidemeister, Murakami, BenheddiCimasoni}. 
In particular, generalizing the fact that the Alexander-Conway polynomial can be constructed using the reduced Burau representation, a multivariable formula is stated by Murakami~\cite[equation (6.10)]{Murakami}, see also Remark~\ref{rem:Murakami}.
\medbreak
In order to describe our main result in this setting, we start by recalling some notions related to the Gassner representation. In fact, since we wish to obtain statements which are valid both in the one variable case and in the multivariable case, we shall work with colored braids and colored links. A \emph{$\mu$-colored link} $L$ is an oriented link $L$ whose components are partitioned into~$\mu$ sublinks $L_1 \cup \ldots \cup  L_\mu$; colored braids are defined similarly: a braid~$\beta$ is \emph{$\mu$-colored} if each of its~$n$ components is assigned (via a surjective map) an element in~$\{1,2,\dots,\mu\}$.  Such a coloring results in two sequences~$c=(c_1,c_2,\dots,c_n)$ and~$c=(c_1',c_2',\dots,c_n')$ of integers: each sequence respectively encodes the colors of the top and bottom boundaries of the resulting \emph{$(c,c')$-braid}. If one fixes such a sequence~$c$, one obtains the group~$B_c$ of $(c,c)$-braids, see Subsection~\ref{sub:Braids} for details. As we shall review in Subsection~\ref{sub:ColoredGassner}, associating to each $n$-stranded $\mu$-colored $(c,c)$-braid its so-called \emph{reduced colored Gassner matrix} produces a homomorphism
$$ \overline{\mathcal{B}}_{(c,c)} \colon B_c \to GL_{n-1}(\Z[t_1^{\pm 1},\ldots,t_\mu^{\pm 1}]).$$
When $\mu=1$, one recovers the reduced Burau matrices~\cite{Burau}, while for $\mu=n$, one retrieves the reduced Gassner matrices~\cite{Birman}. The closure $\widehat{\beta}$ of a 
$(c,c)$-braid $\beta$ is a colored link and, as observed by Birman~\cite[Theorem 3.11]{Birman} and Morton \cite{MortonBraids}, if one uses~$I_k$ to denote the identity matrix of size $k$, then the relation between $\overline{\mathcal{B}}_{(c,c)}(\beta)$ and the Alexander polynomial reads as 
\begin{equation}
\label{eq:AlexanderIntro}
 \Delta_{\widehat{\beta}}(t_1,\ldots,t_\mu) \stackrel{.}{=}
\begin{cases}
\frac{t_1-1}{t_1^n-1}\det(\mathcal{\overline{\mathcal{B}}}_{(c,c)}(\beta)-I_{n-1}) & \mbox{if }  \mu=1, \\
 (t_{c_1}\cdots t_{c_n}-1)\det(\mathcal{\overline{\mathcal{B}}}_{(c,c)}(\beta)-I_{n-1}) & \mbox{if } \mu>1. \\ 
 \end{cases}
 \end{equation}
Finally, we introduce some additional notation. Any $(c,c)$-braid $\beta$ can be decomposed into a product $\prod_{j=1}^{m} \sigma_{i_j}^{\varepsilon_{j}}$, where each $\sigma_{i_j}$ denotes the $i_j$-th generator of the braid group (viewed as an appropriately colored braid) and each $\varepsilon_j$ is equal to $\pm 1$. For each $j$, use~$b_j$ to denote the color of the over-crossing strand in the generator $\sigma_{i_j}^{\varepsilon_{j}}$ and consider the Laurent monomial 
$$ \langle \beta \rangle:=\prod_{j=1}^m t_{b_j}^{-\varepsilon_j}.$$ 
Set $\Lambda_\mu:=\Z[t_1^{\pm 1},\ldots,t_\mu^{\pm 1}]$ and define $g \colon \Lambda_\mu \to \Lambda_\mu$ by extending $\Z$-linearly the group endomorphism of $\Z^\mu=\langle t_1,\ldots,t_\mu \rangle$ which sends $t_i$ to $t_i^2$.  Our main theorem reads as follows:

\begin{theorem}
\label{thm:Main}
Given an $n$-stranded $\mu$-colored $(c,c)$-braid $\beta$, the multivariable potential function of its closure $\widehat{\beta}$ can be described as:
\begin{equation}
\label{eq:Main}
\nabla_{\widehat{\beta}}(t_1,\ldots,t_\mu)=
 (-1)^{n+1} \cdot \frac{1}{t_{c_1} \cdots  t_{c_n} - t^{-1}_{c_1} \cdots  t^{-1}_{c_n}} \cdot \langle \beta \rangle \cdot g(\det(\overline{\mathcal{B}}_{(c,c)}(\beta) - I_{n-1})).
 \end{equation}
\end{theorem}
Theorem~\ref{thm:Main} has three main features. Firstly, it generalizes~\cite[Theorem 3.13]{TuraevKassel} (which deals with the Alexander-Conway polynomial and the Burau representation) to the multivariable case. Secondly, it sign-refines the relation, described in~(\ref{eq:AlexanderIntro}), between the colored Gassner representation and the multivariable Alexander polynomial. Thirdly, it provides an efficient method to compute the multivariable potential function (e.g. by sign refining Morton and Hodgson's algorithm~\cite{MortonHodgson}).

\begin{remark}
\label{rem:Murakami}
As we mentioned above, apart from relating the multivariable potential function to quantum invariants, Murakami also states a formula similar to~(\ref{eq:Main}) in~\cite[equation~(6.10)]{Murakami}. Unfortunately, the sign $(-1)^{n+1}$ does not appear and, in particular, the resulting polynomial is not invariant under the second Markov move. Regardless of this sign issue, Murakami refers to~\cite[equation (2.4)]{Hartley} for a proof of his claim (i.e. for the proof of~\cite[equation~(6.10)]{Murakami}). As it turns out, combining others parts of~\cite{Hartley} with Morton's work~\cite{MortonBraids} does indeed provide a shorter proof of Theorem~\ref{thm:Main} than the one given in Section~\ref{sec:Thm}. This proof is discussed in Appendix~\ref{Appendix} and was generously provided by an anonymous referee.
\end{remark}

The proof of Theorem~\ref{thm:Main} uses a blend of Jiang's axiomatic characterization of $\nabla_L$~\cite{Jiang}, the homological interpretation of the reduced colored Gassner matrices~\cite{KirkLivingstonWang} and ideas of~\cite{TuraevKassel}. More precisely, given a colored link $L$, we use the colored version of the classical theorem of Alexander~\cite{AlexanderClosure} in order to write $L$ as the closure of a colored braid $\beta$. We then associate to~$L$ a rational function $f_L$ which is defined in terms of the reduced colored Gassner representation~$\overline{\mathcal{B}}_{(c,c)}(\beta)$. The fact that this construction provides a well-defined link invariant follows from the colored version of Markov's theorem~\cite{Markov} coupled with homological considerations. Finally, we check that $f_L$ satisfies Jiang's five axioms~\cite{Jiang} which characterize the potential function $\nabla_L$.
\medbreak

The careful reader might have noticed that (up to now) we have only discussed the reduced colored Gassner \emph{matrices}, intentionally avoiding to mention the reduced colored Gassner \emph{representation}. Indeed the latter terminology already refers to a slightly different object which appears in \cite{KirkLivingstonWang, CimasoniConwayGG, CimasoniTuraev, CimasoniConway2Functor}. The aim of the second part of this paper is to clarify the relation between these two objects as well as to provide a more intrinsic description of the reduced colored Gassner matrices. Let us give a brief outline of our results on these issues.

Let $D_n$ denote the $n$ times punctured disk and use $x_1,\ldots,x_n$ to denote the generators of~$\pi_1(D_n,z)$ depicted in Figure~\ref{fig:DiskGeneratorsThesis} (this figure also shows the basepoint $z \in \partial D_n$). Given a sequence $(c_1,\ldots,c_n)$ of integers in $\lbrace 1,\ldots,\mu \rbrace$, consider the regular cover $p \colon \widehat{D}_n \to D_n$ corresponding to the kernel of the homomorphism $\pi_1(D_n) \to \Z^\mu, x_i \mapsto t_{c_i}$. Each braid $\beta$ can be represented by an orientation preserving homeomorphism $h_\beta$ of $D_n$ fixing $\partial D_n$ pointwise. The \emph{unreduced colored Gassner representation}
$$ \mathcal{B}_{(c,c)} \colon B_c \to \operatorname{Aut}_{\Lambda_\mu}(H_1(\widehat{D}_n,p^{-1}(\{ z \})))$$
is obtained by lifting $h_\beta$ to a homeomorphism $\widetilde{h}_\beta \colon \widehat{D}_n \to \widehat{D}_n$ and defining $\mathcal{B}_{(c,c)} (\beta)$ as the induced $\Lambda_\mu$-linear homomorphism on $H_1(\widehat{D}_n,p^{-1}(\{ z \})).$
 
This intrinsic definition contrasts sharply with the coordinate-dependent description of the reduced colored Gassner matrices~\cite{Birman}. Indeed, for $i=1,\ldots,n$, lifts $\widetilde{g}_i$ of the loops $g_i:=x_1 \cdots x_i$ to $\widehat{D}_n$ provide a free basis for $H_1(\widehat{D}_n,p^{-1}(\{ z \}))$ and the \emph{reduced colored Gassner matrix} of $\beta$ is defined as the restriction of $\mathcal{B}_{(c,c)}(\beta)$ to the free $\Lambda_\mu$-module generated by $\widetilde{g}_1,\ldots,\widetilde{g}_{n-1}.$

One might conjecture that the reduced colored Gassner matrices simply represent the $\Lambda_\mu$-automorphism of $H_1(\widehat{D}_n)$ induced by $\widetilde{h}_\beta$. While this is true for $\mu=1$, it cannot hold for $\mu>2$: the former $\Lambda_\mu$-module is not free. For this reason, one considers the localization $\Lambda_S$ of $\Lambda_\mu$ with respect to the multiplicative subset generated by $S=\lbrace 1-t_1,\ldots,1-t_\mu \rbrace$. Indeed, it now turns out that $\Lambda_S \otimes_{\Lambda_\mu} H_1(\widehat{D}_n)$ is free of rank $n-1$ and the \emph{reduced colored Gassner representation}
$$ B_c \to \operatorname{Aut}_{\Lambda_\mu}(\Lambda_S \otimes_{\Lambda_\mu} H_1(\widehat{D}_n))$$
is defined by considering the $\Lambda_S$-linear map induced by $\widetilde{h}_\beta$ on $\Lambda_S \otimes_{\Lambda_\mu} H_1(\widehat{D}_n)$ (note that Kirk-Livingston-Wang~\cite[Definition 2.2]{KirkLivingstonWang} initially defined this representation over the field of fractions $Q$ of $\Lambda_\mu$). In order to state our second result, we introduce one last piece of terminology: we write $\partial \widehat{D}_n \to \partial D_n$ for the restriction of the cover to $\partial D_n$ and we refer to the $\Lambda_S$-linear map induced by $\widetilde{h}_\beta$ on $\Lambda_S \otimes_{\Lambda_\mu} H_1(\widehat{D}_n,\partial \widehat{D}_n)$ as the \emph{map induced by the braid $\beta$}.

Our second result reads as follows.
\begin{theorem}
\label{thm:Second}
Given a $(c,c)$-braid $\beta$ with $n$ strands, the following statements hold:
\begin{enumerate}
\item The map induced by $\beta$ on $\Lambda_S \otimes_ {\Lambda_\mu} H_1(\widehat{D}_n,\partial \widehat{D}_n)$ is represented by the reduced colored Gassner matrix $\overline{\mathcal{B}}_{(c,c)}(\beta)$.
\item The inclusion induced homomorphism $\Phi \colon \Lambda_S \otimes_{\Lambda_\mu} H_1(\widehat{D}_n) \to \Lambda_S \otimes_{\Lambda_\mu} H_1(\widehat{D}_n,\partial \widehat{D}_n)$ intertwines the reduced colored Gassner representation with the map induced by~$\beta$. Furthermore, after tensoring with $Q$, the induced map $\operatorname{id}_Q \otimes \Phi$ is an isomorphism which conjugates the two representations.
\end{enumerate}
\end{theorem}
Summarizing, Theorem~\ref{thm:Second} not only clarifies the relation between the several natural definitions of the ``reduced colored Gassner representation" which have appeared in the literature, it also gives a more intrinsic definition of the reduced colored Gassner matrices which are used in Theorem~\ref{thm:Main}. Conversely, note that Theorem~\ref{thm:Second} can also be viewed as providing a practical way of computing the reduced colored Gassner representation. Finally, note that the second point of Theorem~\ref{thm:Second} implies that Theorem~\ref{thm:Main} also holds for the reduced colored Gassner representation: indeed since both representations are conjugated over $Q$, their determinants agree.
 \medbreak
This paper is organized as follows. Section~\ref{sec:Prelim} reviews colored braids and the colored Gassner representation, Section~\ref{sec:Thm} gives the proof of Theorem~\ref{thm:Main} and Section~\ref{sec:Homology} provides the proof of Theorem~\ref{thm:Second}.

\section*{Acknowledgments}
Both authors wish to thank David Cimasoni for suggesting the project and for several helpful discussions. We are particularly grateful to two anonymous referees: the first pointed us toward~\cite[equation (6.10)]{Murakami}, while the second provided us with a second shorter proof of Theorem~\ref{thm:Main}. The first named author was supported by the NCCR SwissMap funded by the Swiss FNS.

\section{Colored braids and the colored Gassner representation}
\label{sec:Prelim}

\subsection{Colored braids}
\label{sub:Braids}

Following Birman \cite{Birman}, we start by recalling some well-known properties of the braid group. Afterwards, we discuss colored braids, following the conventions of~\cite{ConwayTwistedBurau}.
\medbreak
Let $D^2$ be the closed unit disk in $\mathbb{R}^2.$ Fix a set of $n \geq 1$ punctures $p_1,p_2,\ldots,p_n$ in the interior of $D^2$. We shall assume that the $p_i$ lie in $(-1,1)=\operatorname{Int}(D^2) \cap \mathbb{R}$ and $p_1<p_2<\ldots<p_n.$ A \emph{braid with $n$ strands} is an oriented $n$-component one-dimensional smooth submanifold $\beta$ of the cylinder $D^2 \times [0,1]$ whose oriented boundary is $\bigsqcup_{i=1}^n (p_i \times \lbrace 0) \rbrace \sqcup \left(-\bigsqcup_{i=1}^n (p_i \times \lbrace 1 \rbrace) \right)$, and where the projection to $[0,1]$ maps each component of $\beta$ homeomorphically onto $[0,1]$. Two braids $\beta_1$ and $\beta_2$ are \textit{isotopic} if there is a self-homeomorphism of $D^2 \times [0,1]$ which keeps $\partial (D^2 \times [0,1])$ fixed, such that $h(\beta_1)=\beta_2$. The \textit{braid group} $B_n$ consists of the set of isotopy classes of braids.  The identity element is given by the \textit{trivial braid} $\lbrace p_1,p_2,\dots, p_n \rbrace \times [0,1]$ while the composition of $\beta_1 \beta_2$ consists in gluing $\beta_1$ on top of $\beta_2$ and shrinking the result by a factor $2$ as in Figure \ref{fig:CompositionThesisColor}. 

\begin{figure}[t]
\includegraphics[scale=0.7]{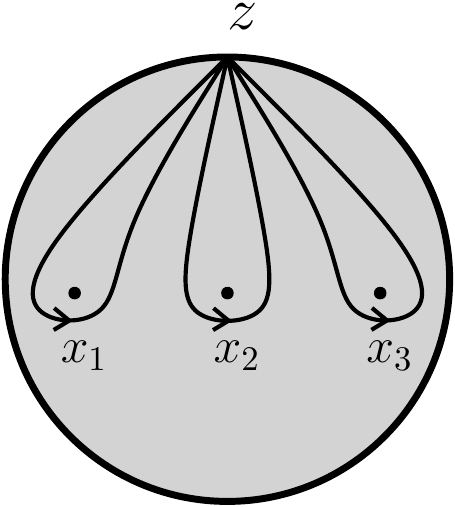}
\caption{The punctured disk $D_3$. }
\label{fig:DiskGeneratorsThesis}
\end{figure}

The braid group~$B_n$ can also be identified with the group of  isotopy classes of orientation-preserving homeomorphisms of~$D_n:=D^2 \setminus \lbrace p_1,\dots, p_n \rbrace$ fixing the boundary pointwise (note that with our conventions, the punctures do not contribute any boundary components: $\partial D_n=\partial D^2$). To understand this fact, first note that a braid $\beta$ induces a deformation retraction of its exterior $X_\beta:=(D^2 \times [0,1]) \setminus \nu \beta$ onto $D_n \times \lbrace 0 \rbrace$. Denoting this retraction by $H_\beta \colon X_\beta \times [0,1] \to X_\beta$, it turns out that the isotopy class (rel $\partial D^2$) of the orientation-preserving homeomorphism $h_\beta \colon D_n \times \lbrace 1 \rbrace \to D_n \times \lbrace 0 \rbrace, x \mapsto H_\beta(x,1)$ depends only on the isotopy class of the braid (see~\cite{Birman} for details).

\begin{figure}[!htb]
\includegraphics[scale=0.7]{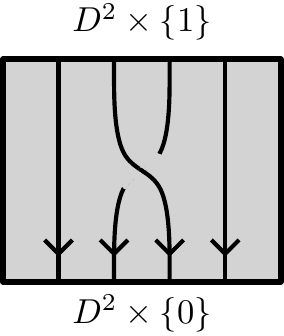}
\caption{The generator $\sigma_2$ of $B_4$. }
\label{fig:GeneratorThesis}
\end{figure}

Either way, $B_n$ admits a presentation with $n-1$ generators $\sigma_1,\sigma_2, \dots, \sigma_{n-1}$ subject to the relations $\sigma_i \sigma_{i+1} \sigma_i=\sigma_{i+1} \sigma_i \sigma_{i+1}$ for each $i$, and $\sigma_i \sigma_j = \sigma_j \sigma_i$ if $|i-j|>2$. Topologically, the generator $\sigma_i$ is the braid whose $i$-th component passes over the $i+1$-th component as shown in Figure~\ref{fig:GeneratorThesis}. Sending a braid to its underlying permutation produces a surjection from the braid group into the symmetric group. The kernel $P_n$ of this map is called the \textit{pure braid group}. 

\begin{remark}
\label{rem:ConventionGenerator}
Although we have chosen to follow Birman's convention regarding the topological interpretation of $\sigma_i$~\cite{Birman}, this convention is by no means standard: the opposite convention is also widespread in the literature. To only name two examples, both Morton's article~\cite{MortonBraids} and Birman and Brendle's survey~\cite{BirmanBrendle} assume that $\sigma_i$ is represented by the braid whose $i$-th component passes \emph{under} the $i+1$-th component.
\end{remark}

Fix a base point $z$ of $D_n$ in $\partial D_n$ and denote by $x_i$ the simple loop based at $z$ turning once around $p_i$ counterclockwise for $i=1,2,\dots, n$ as in Figure \ref{fig:DiskGeneratorsThesis}. The group $\pi_1(D_n,z)$ can then be identified with the free group $F_n$ on the $x_i.$ 
If $h_\beta$ is a homeomorphism of $D_n$ representing a braid $\beta$, then the induced automorphism $h_{\beta*}$ of the free group $F_n$ only depends on $\beta$. It follows from the way we compose braids that $h_{(\gamma \beta)*}=h_{\beta*}h_{\gamma*}$, and the resulting \emph{anti}-representation $B_n \to \operatorname{Aut}(F_n)$ can be explicitly described by
$$ (h_{\sigma_i})_*x_j=
\begin{cases}
x_i x_{i+1} x_i^{-1}  & \mbox{if }  j=i, \\
 x_i                            & \mbox{if }  j=i+1, \\ 
x_j                             & \mbox{otherwise. } \\ 
 \end{cases} $$
The \textit{closure} of a braid $\beta$ is the link $\widehat{\beta}$ obtained from $\beta$ by adding parallel strands in $S^3 \setminus (D^2 \times [0,1])$ as in Figure \ref{fig:ClosureThesis}. While Alexander's theorem~\cite{AlexanderClosure} ensures that every link can be obtained as the closure of a braid, the correspondence between braids and links is not one-to-one: non-isotopic braids can have isotopic closures. As we shall recall below, Markov's theorem~\cite{Markov} describes a complete set of moves which relates  braids whose closures are isotopic.

\begin{remark} \label{rem:proofAlexander} 
In fact, a close inspection of the proof of Alexander's theorem leads to the following refined statement. If an oriented link contains a braid~$\alpha$ in a small cylinder, then it can be obtained as the closure of a braid which contains $\alpha$ in a small cylinder (with orientations as shown in Figure~\ref{fig:ClosureThesis} below).
\end{remark}

 \begin{figure}[!htb]
\centering
\includegraphics[scale=0.7]{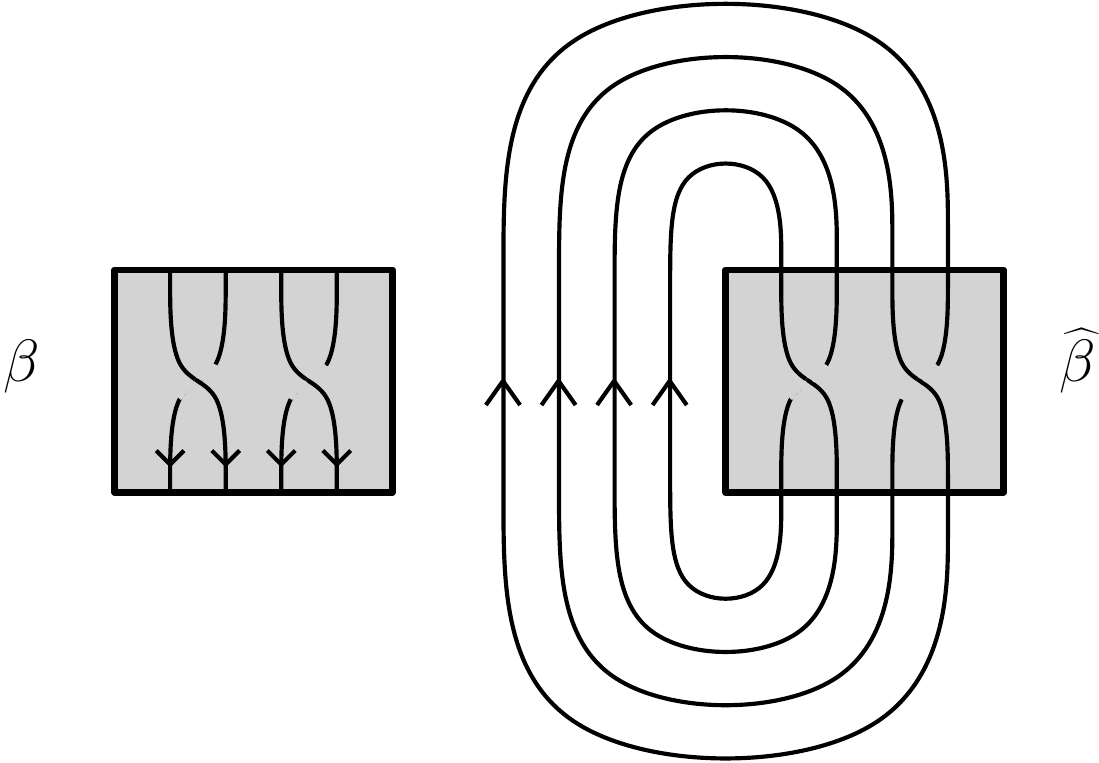}
\caption{The closure of a braid.}
\label{fig:ClosureThesis}
\end{figure}

\begin{figure}[t]
\centering
\includegraphics[scale=0.7]{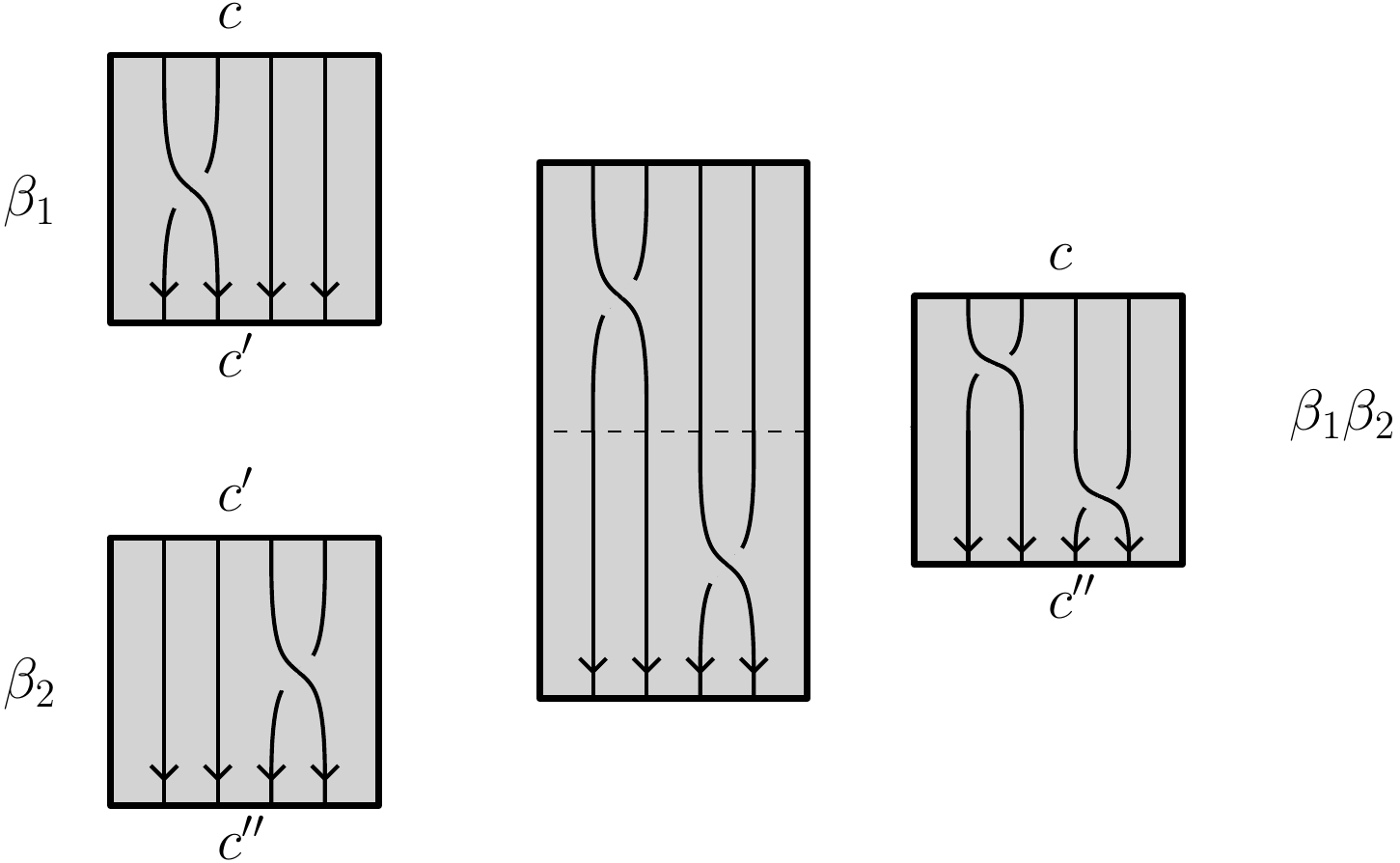}
\caption{A $(c,c')$-braid $\beta_1$, a $(c',c'')$-braid $\beta_2$ and their composition, the $(c,c'')$-braid $\beta_1\beta_2$.}
\label{fig:CompositionThesisColor}
\end{figure}

A braid~$\beta$ is{\em~$\mu$-colored\/} if each of its components is assigned (via a surjective map) an integer in~$\{1,2,\dots,\mu\}$ (which we call a \emph{color}). A~$\mu$-colored braid induces a coloring on the punctures of $D^2 \times \lbrace 0, 1 \rbrace$. For emphasis, we shall denote the resulting punctured disks by~$D_c$ and~$D_{c'}$, and call a~$\mu$-colored braid a~\emph{$(c,c')$-braid}, where~$c$ and~$c'$ are the sequences of~$ 1, 2,\dots,\mu$ induced by the coloring of the braid. Two colored braids are isotopic if the underlying isotopy is color preserving, and we shall denote by~$\id_c$ the isotopy class of the trivial~$(c,c)$-braid. The composition of a $(c,c')$-braid $\beta_1$ with a $(c',c'')$-braid $\beta_2$ is the $(c,c'')$-braid $\beta_1 \beta_2$ depicted in Figure~\ref{fig:CompositionThesisColor}. Thus, for any sequence~$c$, the set~$B_c$ of isotopy classes of~$(c,c)$-braids is a group which interpolates between the braid group $B_n=B_{(1,1,\dots, 1)}$ and the pure braid group~$P_n=B_{(1,2,\dots,n)}$. Additionally, we shall often use the map $i_{c_{n+1}} \colon B_c \hookrightarrow B_{(c_1, \ldots, c_n, c_{n+1})}$ which sends $\alpha$ to the disjoint union of $\alpha $ with a trivial strand of color $c_{n+1}$, see Figure~\ref{fig:Inclusion}. Here, note that $c_{n+1}$ can very well be equal to one of the~$n$ first~$c_i$'s.

Finally, the closure of a $\mu$-colored braid $\beta$ is the $\mu$-colored link $\widehat{\beta}$ obtained from $\beta$ by adding colored parallel strands in $S^3 \setminus (D^2 \times [0,1]).$ We refer to~\cite[Theorem 3.3]{Murakami} for the colored version of Alexander's theorem (which states that every colored link can be obtained as the closure of a colored braid) and instead focus on the colored version of Markov's theorem, referring to~\cite[Theorem 3.5]{Murakami} for details.

\begin{proposition}
\label{prop:MarkovColored}
Two $(c,c)$-braids have isotopic closures if and only if they are related by a sequence of the following moves and their inverses:	
\begin{enumerate}
\item replace $\alpha \beta$ by $\beta \alpha$, where $\alpha$ is a $(c,c')$-braid and $\beta$ is a $(c',c)$-braid,
\item replace $\alpha$ by $\sigma^{\varepsilon}_n i_{c_n}(\alpha)$, where $\alpha$ is a $(c,c)$-colored braid with $n$ strands, $\sigma_n$ is viewed as a $((c_1, \ldots , c_n, c_n),(c_1, \ldots , c_n, c_n))$-braid, and $\varepsilon$ is equal to $\pm 1$.
\end{enumerate}
\end{proposition}

\begin{figure}[!h]
\centering
\includegraphics[scale=0.7]{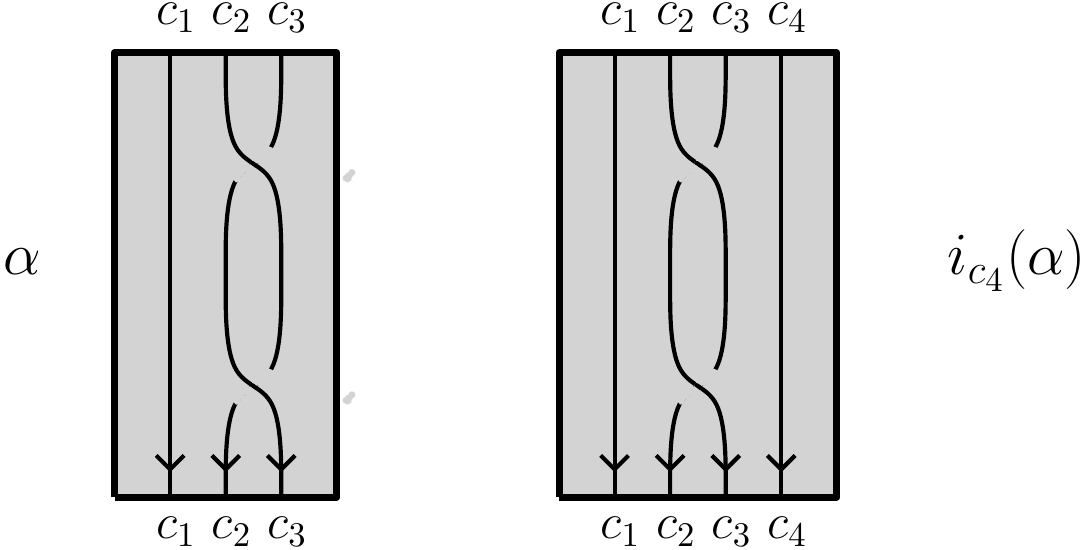}
\caption{An example of the inclusion map $i_{c_4}$.}
\label{fig:Inclusion}
\end{figure}

\subsection{The colored Gassner representation}
\label{sub:ColoredGassner}

In this subsection, we review the homological definition of the unreduced colored Gassner representation (following~\cite{KirkLivingstonWang, TuraevFaithful, ConwayTwistedBurau}) and of the reduced colored Gassner matrices (following~\cite{Birman,MortonBraids, Murakami, ConwayTwistedBurau}). A more leisurely exposition can also be found in~\cite[Chapter 9]{ConwayThesis}. It must however be mentioned that our conventions are actually closest to those used in~\cite{BenAribiConway}; in particular the unreduced colored Gassner representation is in fact an \emph{anti}-representation. Other appearances of the colored Gassner representation include work of Penne~\cite{Penne1, Penne2}.
\medbreak
Fix a sequence $(c_1,\ldots,c_n)$ of elements in $\lbrace 1,\ldots, \mu \rbrace$ and a basepoint $z$ of the punctured disk~$D_c$ which lies in $\partial D_c$. Consider the map $\psi_c \colon \pi_1(D_c) \rightarrow \Z^\mu=\langle t_1,\dots,t_\mu \rangle$ which sends each~$x_i$ to $t_{c_i}$. Let $\widehat{D}_c \rightarrow D_c$ be the regular cover corresponding to $\ker(\psi_c)$ and let $P$ be the fiber over $z$. The homology groups of $\widehat{D}_c $ are naturally modules over $\Lambda_\mu=\Z[t_1^{\pm 1},\dots,t_\mu^{\pm 1}]$. Given a homeomorphism $h_\alpha \colon D_c \rightarrow D_{c'}$ representing a $(c,c')$-braid $\alpha$, one can check that $h_\alpha $ lifts to a unique homeomorphism $\widetilde{h}_\alpha \colon \widehat{D}_c  \rightarrow \widehat{D}_{c'}$ fixing $P=P'$ pointwise. Taking the induced map on homology produces a well-defined $\Lambda_\mu$-homomorphism
 $$ \mathcal{B}_{(c,c')}(\alpha)  \colon H_1(\widehat{D}_c ,P) \rightarrow H_1(\widehat{D}_{c'},P').$$
In the case where $c=c'$, we obtain a map $B_c \to \operatorname{Aut}_{\Lambda_\mu}(H_1(\widehat{D}_c,P))$ which we call the \textit{unreduced colored Gassner representation}. When $\mu=1$, the unreduced colored Gassner representation recovers the unreduced Burau representation of the braid group $B_n$ while if $\mu=n$, we retrieve the unreduced Gassner representation of the pure braid group described in~\cite{Birman}, see~\cite{ConwayTwistedBurau} and~\cite[Chapter 9]{ConwayThesis} for details.

Since the proof of the following proposition can be found in~\cite{ConwayTwistedBurau}, we only sketch it here.

\begin{proposition}
\label{prop:GassnerCocycle}
Given a $(c,c')$-braid $\beta$ and a $(c',c'')$-braid $\gamma$, we have
$$  \mathcal{B}_{(c,c'')}(\beta \gamma)=\mathcal{B}_{(c',c'')}(\gamma)\mathcal{B}_{(c,c')}(\beta).$$
In particular, $\mathcal{B}_{(c',c)}(\beta^{-1})=\mathcal{B}_{(c,c')}(\beta)^{-1}$ and, restricting to $(c,c)$-braids, $\mathcal{B}_{(c,c)}$ is an anti-representation.
\end{proposition}
\begin{proof}
Fix an arbitrary lift of $z$ to $\widehat{D}_c$. Since the lift of $h_{\alpha \beta}$ coincides with the lift of $h_{\beta}\circ h_{\alpha}$, the first assertion follows. The second and third statements are immediate consequences of the first.
\end{proof}

Note that the homology $\Lambda_\mu$-module $H_1(\widehat{D}_c,P)$ is free of rank $n$: it is easily shown that lifts $\widetilde{x}_1,\ldots,\widetilde{x}_n$ of the $x_1,\ldots,x_n$ provide a $\Lambda_\mu$-basis~\cite[Lemma 2.2]{ConwayTwistedBurau}. With respect to this basis, the transpose of the matrix for the unreduced colored Gassner representation of the generator~$\sigma_i$ (viewed as a $(c,c')$-braid) can be found in~\cite[Example 3.5]{ConwayTwistedBurau}.

Next following~\cite{Birman} and~\cite[Section 3 (c)]{ConwayTwistedBurau}, we deal with the reduced colored Gassner matrices. Instead of working with the free generators~$x_1,x_2\dots,x_n$ of~$\pi_1(D_c),$ one can consider the elements~$g_1, g_2, \dots,g_n,$ defined by~$g_i:=x_1 x_2 \cdots x_i$. For $i=1,\ldots, n$, let~$\widetilde{g}_i$ be the lift of~$g_i$ to~$\widehat{D}_c$ starting at a fixed lift of $z$. One obtains the splitting 
$$H_1(\widehat{D}_c,P) = \bigoplus_{i=1}^{n-1} \Lambda_\mu \widetilde{g}_i \oplus \Lambda_\mu \widetilde{g}_n.$$
 As~$g_n$ is always fixed by the action of the braid group, its lift~$\widetilde{g}_n$ is fixed by the lift~$\widetilde{h}_\beta$ of a homeomorphism $h_\beta$ representing a colored braid $\beta$. 

\begin{definition}
\label{def:ReducedColoredGassnerAlgebraic}
The \emph{reduced colored Gassner matrix} of a $(c,c')$-braid~$\beta$ is the restriction~$\overline{\mathcal{B}}_{(c,c')}(\beta)$ of the unreduced colored Gassner map to the free~$\Lambda_\mu$-module of rank $(n-1)$ generated by $\widetilde{g}_1,\dots,\widetilde{g}_{n-1}$.
\end{definition} 

As an immediate consequence of Definition~\ref{def:ReducedColoredGassnerAlgebraic}, observe that the reduced colored Gassner matrices satisfy the relations described in Proposition~\ref{prop:GassnerCocycle}. Furthermore, using~$\mathcal{B}_{(c,c')}(\beta)$ to denote the matrix of the unreduced colored Gassner representation of a braid~$\beta$ with respect to the basis $\widetilde{g}_1,\ldots,\widetilde{g}_n$, it follows that 
\begin{equation}
\label{eq:ReducedGassnerFox}
\mathcal{B}_{(c,c')}(\beta)=
\begin{pmatrix} \overline{\mathcal{B}}_{(c,c')}(\beta) & 0 \\ v & 1 \end{pmatrix} 
\end{equation}
for some length $(n-1)$ row vector~$v$. In particular, as explained in~\cite[Example 3.10]{ConwayTwistedBurau}, the reduced colored Gassner matrix of the generator~$\sigma_i$ (viewed as a $(c,c')$-braid) is given by
\begin{align}
\label{eq:MatricesReduced}
 \overline{\mathcal{B}}_{(c,c')}(\sigma_i)
 &=I_{i-2} \oplus 
\begin{pmatrix}
1 & t_{c'_{i+1}} & 0 \\
0 & -t_{c'_{i+1}}  & 0 \\
0 & 1 & 1
\end{pmatrix}
 \oplus I_{n-i-2}
 \end{align}
for~$1<i<n-1$, and for~$\sigma_1$ and~$\sigma_{n-1}$ by
\begin{align*}
 \overline{\mathcal{B}}_{(c,c')}(\sigma_1)
 &= \begin{pmatrix}  -t_{c'_2}  &0 \\
 1 & 1 \end{pmatrix}
 \oplus I_{n-3}, \\
 \overline{\mathcal{B}}_{(c,c')}(\sigma_{n-1})
 &=I_{n-3} \oplus 
\begin{pmatrix}
1 & t_{c'_n}  \\
0 & -t_{c'_n}   \\
\end{pmatrix}.
 \end{align*}
We conclude this section by emphasizing once more that the description of the reduced colored Gassner \emph{matrices} given here differs from the ``reduced colored Gassner \emph{representation}" of~\cite{KirkLivingstonWang, CimasoniTuraev, CimasoniConwayGG}. The relation between these constructions will be clarified in Section~\ref{sec:Homology}.

\section{The multivariable potential function}
\label{sec:Thm}
In this section, we prove Theorem~\ref{thm:Main} by giving a construction of the multivariable potential function which involves the reduced colored Gassner matrices. As we mentioned in the introduction, the proof uses a blend of Jiang's axiomatic characterization of $\nabla_L$~\cite{Jiang}, the homological interpretation of the reduced colored Gassner matrices and ideas of Kassel-Turaev~\cite[Section 3.4]{TuraevKassel}. 

The proof decomposes into three steps: first, given a link $L$, we define a rational function $f_L$, secondly we show that $f_L$ is a link invariant (see Proposition~\ref{prop:Invariance}) and thirdly we show that $f_L$ coincides with the multivariable potential function $\nabla_L$, proving~Theorem~\ref{thm:Main}. Subsection~\ref{sub:Invariantf} deals with the first two steps while Subsection~\ref{sub:Proof} is concerned with the third. Finally, note that an alternative proof of Theorem~\ref{thm:Main} is presented in Appendix~\ref{Appendix}.

\subsection{The invariant $f$}
 \label{sub:Invariantf}

Any $(c,c)$-braid $\beta$ can be decomposed into a product of generators $\prod_{j=1}^{m} \sigma_{i_j}^{\varepsilon_{j}}$, where each $\sigma_{i_j}$ denotes the $i_j$-th generator of the braid group (viewed as an appropriately colored braid) and each $\varepsilon_j$ is equal to $\pm 1$. For each $j$, use~$b_j$ to denote the color of the over-crossing strand in the generator $\sigma_{i_j}^{\varepsilon_{j}}$ and consider the Laurent polynomial 
$$ \langle \beta \rangle:=\prod_{j=1}^m t_{b_j}^{-\varepsilon_j}.$$ 
Finally, define $g \colon \Lambda_\mu \to \Lambda_\mu$ by extending $\Z$-linearly the group endomorphism of $\Z^\mu=\langle t_1,\ldots,t_\mu \rangle$ which sends $t_i$ to $t_i^2$.  
\begin{definition} 
\label{def:Invariantf}
For any $(c,c)$-braid $\beta$ with~$n$ strands, set
$$ f(\beta):=  (-1)^{n+1} \cdot \frac{1}{t_{c_1} \cdots  t_{c_n} - t^{-1}_{c_1} \cdots  t^{-1}_{c_n}} \cdot \langle \beta \rangle \cdot g(\det(\Grcc(\beta) - I_{n-1})).$$
\end{definition}

In order to define $f$ on a colored link $L$, proceed as follows: use the colored Alexander theorem in order to write $L$ as the closure of a $(c,c)$-braid $\beta$ and set 
$$f_L:=f(\beta).$$
Observe that $f$ is only well-defined provided it takes the same value on colored braids whose closures are isotopic. The proof of this result will be given in Proposition~\ref{prop:Invariance}. However, accepting this fact for the time being, we provide some sample computations.

\begin{example}
\label{ex:ExampleHopf} 
Set $c=(1,2)$ and view the $2$-colored positive Hopf link $H$ as the closure of the $2$-stranded $(c,c)$-braid $\sigma^{-2}_1$. Since $\langle \sigma^{-2}_1 \rangle$ is given by $t_1 t_2$ and  $\mathcal{B}_{(c,c)}(\sigma^{-2}_1)=t^{-1}_1 t^{-1}_2$ (here we used~(\ref{eq:MatricesReduced}) and Proposition~\ref{prop:GassnerCocycle}), we deduce from Definition~\ref{def:Invariantf} that
$$ f_H = (-1)^3 \frac{t_1 t_2 }{t_1 t_2 - t^{-1}_1 t^{-1}_2} (t^{-2}_1 t^{-2}_2 - 1) = 1. $$
\end{example}

Next, we give a slightly more involved example:

\begin{example}
\label{ex:InvariantChain}
Set $c=(1,2,3)$ and view the link $L$ depicted in Figure~\ref{fig:Chain} as the closure of the $3$-stranded $(c,c)$-braid~$\sigma^{-2}_1 \sigma^{-2}_2$. Using~(\ref{eq:MatricesReduced}) and Proposition~\ref{prop:GassnerCocycle}, we can compute $\mathcal{B}_{(c,c)}(\sigma^{-2}_1 \sigma^{-2}_2)$. After subtracting the identity, taking the determinant and applying $g$, we obtain $1-t^{-2}_2 - t^{-2}_1 t^{-2}_2 t^{-2}_3 + t^{-2}_1 t^{-4}_1 t^{-2}_1$. Moreover, since $\langle \sigma^{-2}_1 \sigma^{-2}_2 \rangle$ is equal to $t_1 t^2_2 t_3$, Definition~\ref{def:Invariantf} implies that
$$ f_L = (-1)^4 \frac{t_1 t^2_2 t_3}{t_1 t_2 t_3 - t^{-1}_1 t^{-1}_2 t^{-1}_3} (1 - t^{-2}_2 - t^{-2}_1 t^{-2}_2 t^{-2}_3 + t^{-2}_1 t^{-4}_1 t^{-2}_1) = t_2 - t^{-1}_2. $$
\end{example}

\begin{figure}[!htb]
\includegraphics[scale=0.5]{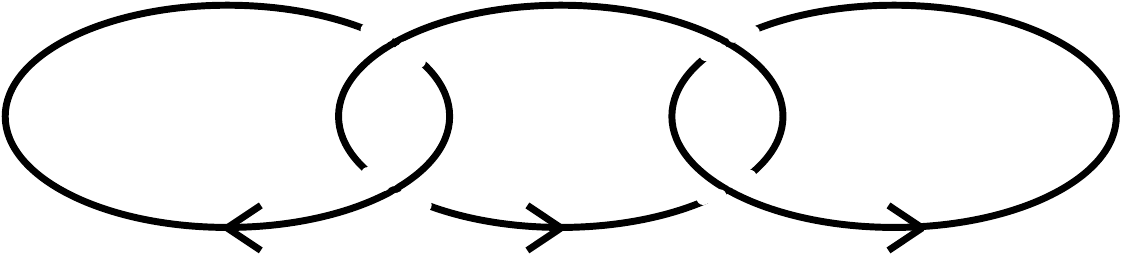}
\caption{The link $L$ used in Example~\ref{ex:InvariantChain}.}
\label{fig:Chain}
\end{figure}

In order to prove the invariance of $f$, we shall show that it is invariant under the colored Markov moves described in Proposition~\ref{prop:MarkovColored}. To do so, we start with a preliminary lemma. Given a $(c,c)$-braid $\beta$, recall from~(\ref{eq:ReducedGassnerFox}) that in the basis $\widetilde{g}_1,\ldots,\widetilde{g}_n$ of $H_1(\widehat{D}_c,P)$, the unreduced colored Gassner matrix of $\beta$ can be written as
$$ \mathcal{B}_{(c,c)}(\beta)=\begin{pmatrix} 
\Grcc(\beta) & 0 \\
v & 1 \end{pmatrix}, $$
where $v$ is a row vector. The next lemma shows that this vector can be expressed in terms of the reduced colored Gassner matrix.

\begin{lemma}
\label{lem:NoInfo}
Given a $(c,c)$-braid $\beta$ with $n$ strands, use $r_i$ to denote the i$^{\text{th}}$ row of the matrix $\Grcc(\beta) - I_{n-1}$. The following equality holds:
\begin{equation}
\label{eq:NoInfo}
\sum_{i=1}^{n-1} (t_{c_1}  \cdots t_{c_i} - 1)r_i  = -(t_{c_1} \ldots t_{c_n} - 1)v.
\end{equation}
\end{lemma}

\begin{proof}
Fix a basepoint $z$ in $\partial D_c$ and let $P$ be its fiber in the cover $\widehat{D}_c \to D_c$. Let $h_\beta$ be a self-homeomorphism of $D_c$ representing $\beta$, fix an arbitrary lift of $z$ to $\widehat{D}_c$ and let $\widetilde{h}_{\beta} \colon \widehat{D}_c \rightarrow \widehat{D}_c$ be the lift of $h_\beta$ fixing $P$ pointwise. Using $\partial$ to denote the connecting homomorphism in the long exact sequence of the pair $(\widehat{D}_c,P)$, the following diagram commutes by naturality of the long exact sequence in homology:
$$ \xymatrix@R1cm{ H_1(\widehat{D}_c, P) \ar[d]_{\mathcal{B}_{(c,c)}(\beta)} \ar[r]^{\partial} & H_0(P) \ar[d]^{(\widetilde{h}_{\beta})_*} \\
H_1(\widehat{D}_c, P) \ar[r]^{\partial}  &  H_0(P). }$$
Since $\widetilde{h}_{\beta}$ fixes $P$ pointwise, it induces the identity on degree zero homology. With respect to the basis $\widetilde{g}_1,\ldots,\widetilde{g}_n$ of $H_1(\widehat{D}_c, P)$ the connecting homomorphism $\partial$ is represented by the $1 \times n$ matrix $(t_{c_1} - 1, t_{c_1} t_{c_2} - 1, \ldots , t_{c_1} t_{c_2} \ldots t_{c_n} - 1)$. Writing out explicitly the equation $ \partial \circ  \mathcal{B}_{(c,c)}(\beta) = \partial$ yields~(\ref{eq:NoInfo}), concluding the proof of the lemma.
\end{proof}

Given a sequence $c=(c_1,\ldots,c_n)$ of integers in $\lbrace 1,\ldots ,\mu \rbrace$, recall that $i_{c_n} \colon B_c \hookrightarrow B_{(c_1, \ldots, c_n, c_n)}$ denotes the natural inclusion which sends $\alpha$ to the disjoint union of $\alpha $ with a trivial strand of color $c_n$. We can now prove the main result of this subsection, namely the invariance of $f$ under the colored Markov moves.

\begin{proposition}
\label{prop:Invariance}
The rational function $f$ is invariant under both colored Markov moves. More precisely, we have the following equalities:
\begin{enumerate}
\item $f(\alpha \beta) = f(\beta \alpha)$ for all $(c,c')$-braids $\alpha$ and all $(c',c)$-braids $\beta$.
\item $f(\alpha) = f( \sigma_{n}^{\varepsilon}i_{c_n}(\alpha))$ for all $n$-stranded $(c,c)$-braids $\alpha$, where the $n$-th generator~$\sigma_n$ of~$B_{n+1}$ is viewed as a $((c_1, \ldots ,c_n, c_n),(c_1, \ldots ,c_n, c_n))$-braid and $\varepsilon$ is equal to $\pm 1$.
\end{enumerate}
\end{proposition}
\begin{proof}

To show the first statement, given a $(c,c')$-braid $\alpha$ and a $(c',c)$-braid $\beta$, our goal is to show that $f(\alpha \beta)$ and $f(\beta \alpha)$ coincide. Since $\langle \alpha \beta \rangle= \langle \beta \alpha \rangle$, this clearly reduces to showing 
\begin{equation}
\label{eq:Goal0}
\det(\Grcc(\alpha \beta) - I_{n-1}) = \det(\Grcpcp(\beta \alpha) - I_{n-1}).
\end{equation}
Using the equality $\alpha \beta=\alpha \beta \alpha \alpha^{-1}$ and Proposition~\ref{prop:GassnerCocycle}, we deduce that $\Grcc(\alpha \beta) - I_{n-1}$ is equal to $
\overline{\mathcal{B}}_{(c,c')}(\alpha)^{-1}(\Grcpcp(\beta \alpha) - I_{n-1})\Grccp(\alpha) $. This immediately implies~(\ref{eq:Goal0}).

To prove the second statement, fix a $(c,c)$-braid $\alpha$, set $\varepsilon = +1$ (the case $\varepsilon = -1$ is treated identically), and write $c'$ for $(c_1, \ldots ,c_n, c_n)$. Our goal is to show that $f(\alpha) = f( \sigma_{n}i_{c_n}(\alpha))$. Using Definition~\ref{def:Invariantf} and the equality $\big \langle \sigma_{n}i_{c_n}(\alpha) \big \rangle=t_{c_n}^{-1} \big \langle \alpha \big \rangle$, it is enough to show that
\begin{equation}
\label{eq:Goal}
\frac{g(\det(\Grcc(\alpha)-I_{n-1}))}{t_{c_1} \cdots  t_{c_n} - t^{-1}_{c_1} \cdots  t^{-1}_{c_n}} 
= \frac{-t_{c_n}^{-1} \cdot g\left( \Grcpcp( \sigma_{n}i_{c_n}(\alpha) ) - I_n \right)}{t_{c_1} \cdots t_{c_{n-1}} t_{c_n}^2 - t^{-1}_{c_1} \cdots t^{-1}_{c_{n-1}} t^{-2}_{c_n}}.
\end{equation}
Our aim is now to compare the determinants of $\Grcpcp(\sigma_n i_{c_n}(\alpha)) - I_{n}$ and of $\Grcc(\alpha) - I_{n-1}$. To do so, we start by investigating $\Grcpcp(i_{c_n}(\alpha))$.  Since $h_{i_{c_n}(\alpha)}(\widetilde{g}_i)=h_\alpha(\widetilde{g}_i)$ for $i=1,\ldots,n$, we deduce that $\Grcpcp(i_{c_n}(\alpha))$ is given by $\left(\begin{smallmatrix} \Grcc(\alpha) & 0 \\ v & 1 \end{smallmatrix}\right)$, where $v$ is a length $(n-1)$ row vector. The goal is now to express the determinant of $\Grcpcp( \sigma_{n} i_{c_n}(\alpha)) - I_{n}$ in terms of the determinant of $\Grcc(\alpha) - I_{n-1}$. To that end, we write $\overline{\mathcal{B}}_{(c,c)}(\alpha)$ as $\left(\begin{smallmatrix} B & b_1\\ b_2 & a  \end{smallmatrix}\right)$ and $\Grcpcp(i_{c_n}(\alpha))$ as
\begin{equation}
 \label{matriceexplicitee1}
\Grcpcp(i_{c_n}(\alpha)) = \begin{pmatrix} 
B & b_1 & 0\\
b_2 & a & 0\\
v_1 & v_2 & 1
\end{pmatrix},
\end{equation}
where $B$ is a square matrix of size $n-2$, $b_1$ is a $(n-2) \times 1$ matrix, $b_2$ and $v_1$ are $1 \times (n-2)$ matrices, and $a$ and $v_2$ belong to~$\Lambda_{\mu}$. Using successively Proposition~\ref{prop:GassnerCocycle} and~(\ref{eq:MatricesReduced}), we deduce that
$$  \Grcpcp( \sigma_{n}i_{c_n}(\alpha) ) - I_n =
\begin{pmatrix} 
B-I_{n-2} & b_1 & t_{c_n} b_1\\
b_2 & a -1 & t_{c_n} a\\
v_1 & v_2 & t_{c_n}(v_2 - 1) - 1
\end{pmatrix}. $$
Our plan is to use Lemma~\ref{lem:NoInfo} and a sequence of elementary operations in order to remove the vectors $v_1$ and $v_2$. Firstly, we subtract the second-to-last column multiplied by $t_{c_n}$ to the last column. Secondly, using $A_i$ to denote the rows of the resulting matrix, we multiply the last row of this matrix by $(t_{c_1} \cdots t_{c_n} -~1)$ and add to it $\sum_{i=1}^{n-1}(t_{c_1}\cdots t_{c_i}-1)A_i$. Using Lemma~\ref{lem:NoInfo}, the result of these two operations is
$$ X:=\begin{pmatrix} 
B-I_{n-2} & b_1 & 0\\
b_2 & a-1 & t_{c_n}\\
0 & 0 & e
\end{pmatrix},$$
where $e$ stands for  $(1-t_{c_1} \cdots t_{c_{n-1}} t_{c_n}^2)$. Notice that the second operation we performed yields a factor of $(t_{c_1} \cdots t_{c_n} -1)^{-1}$ to the determinant; more precisely, $\det(X)=(t_{c_1} \cdots t_{c_n} -1) \det( \Grcpcp( \sigma_{n}i_{c_n}(\alpha) ) - I_n)$. Combining these observations and computing $\det(X)$ by expanding along the last row, we obtain
$$ \det(\Grcpcp(\sigma_{n}i_{c_n}(\alpha))-I_{n}) = \frac{1-t_{c_1} \cdots t_{c_{n-1}} t_{c_n}^2}{t_{c_1} \cdots t_{c_n} - 1} \det(\Grcc(\alpha)-I_{n-1}).$$
Plugging this equality into the right hand side of~(\ref{eq:Goal}), the verification of the second Markov move reduces to checking the following equality:
$$ \frac{g(\det(\Grcc(\alpha)-I_{n-1}))}{t_{c_1} \cdots  t_{c_n} - t^{-1}_{c_1} \cdots  t^{-1}_{c_n}}=
\frac{-t_{c_n}^{-1}  g(\det(\Grcc(\alpha)-I_{n-1})) }{t_{c_1} \cdots t_{c_{n-1}} t_{c_n}^2 - t^{-1}_{c_1} \cdots t^{-1}_{c_{n-1}} t^{-2}_{c_n}} g\left( \frac{1-t_{c_1} \cdots t_{c_{n-1}} t_{c_n}^2}{t_{c_1} \cdots t_{c_n} - 1} \right). $$
Simplifying the $g(\det(\Grcc(\alpha)-I_{n-1})$, this latter equation can easily be verified to hold.
\end{proof}

\subsection{Proof of Theorem~\ref{thm:Main}}
\label{sub:Proof}

By Proposition~\ref{prop:Invariance}, we know that $f_L$ is a link invariant. In order to prove Theorem~\ref{thm:Main} (which states that $f_L$ is equal to the multivariable potential function $\nabla_L$) we shall use Jiang's characterization theorem~\cite{Jiang} which states that $\nabla_L$ is uniquely determined by the following set of five local relations: 
\begin{itemize}
\item [(R1)] $\nabla_H=1$, where $H$ is the positive Hopf link.
\item [(R2)] $\nabla_{L \sqcup U}=0$, where $L \sqcup U$ denotes the disjoint union of $L$ and a trivial knot $U$.
\item [(R3)] $\nabla_{L'}=(t_i-t_i^{-1})\nabla_{L_0}$, where $L'$ is obtained from $L_0$ by the local operation given by
\begin{figure}[!htb]
\centering
\includegraphics[scale=0.5]{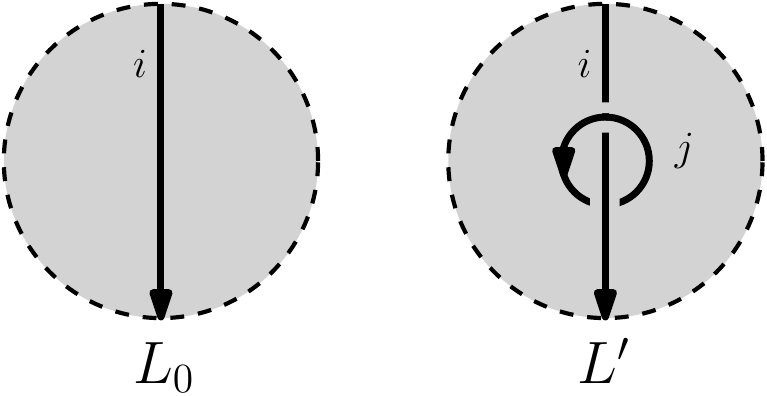}
\end{figure}
\item [(R4)] $\nabla_{L_{++}}+\nabla_{L_{--}}=(t_it_j-t_i^{-1}t_j^{-1})\nabla_{L_0}$, where $L_{++}, L_{--}$ and $L_0$ differ by the local relation
\begin{figure}[!htb]
\centering
\includegraphics[scale=0.5]{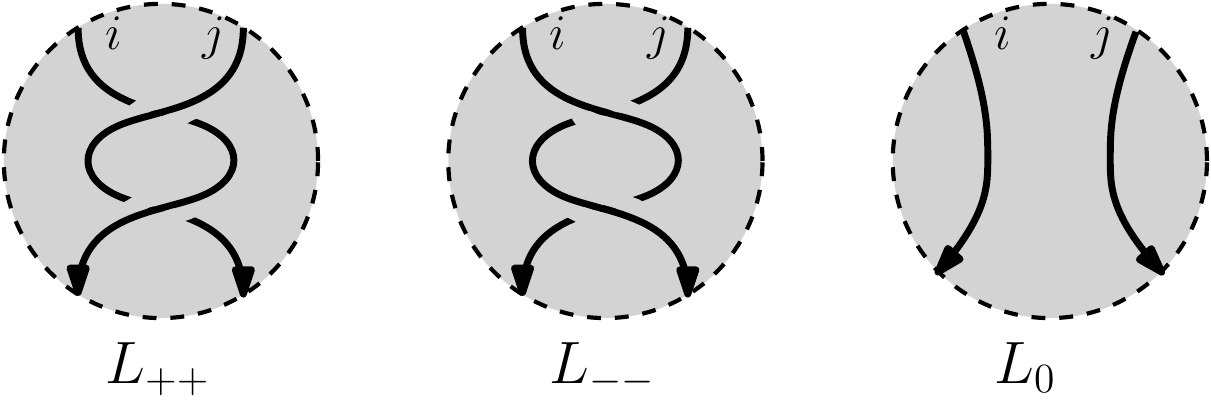}
\end{figure}
\item [(R5)] $
(t_i^{-1}t_j^{-1}-t_it_j)(\nabla_{L(1)}+\nabla_{L(2)})
+(t_jt_k-t_j^{-1}t_k^{-1})(\nabla_{L(3)}+\nabla_{L(4)})$ \medbreak
$\qquad \qquad  \qquad \qquad \qquad \qquad \qquad \qquad \qquad+
(t_it_k^{-1}-t_i^{-1}t_k)(\nabla_{L(5)}+\nabla_{L(6)})=0,$ \medbreak
where $L(1),L(2),L(3),L(4),L(5)$ and $L(6)$ differ by the local operation

\begin{figure}[!htb]
\centering
\includegraphics[scale=0.5]{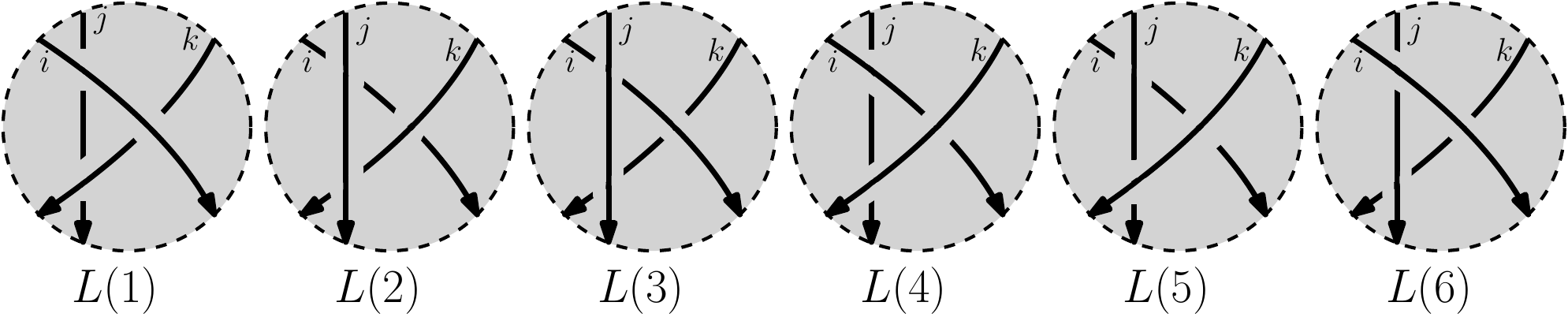}
\end{figure}
\end{itemize}

Since each of Jiang's axioms is written in terms of local relations, we wish to find braids whose closures realize these relations. Even though the end result is independent of such choices (thanks to Proposition \ref{prop:Invariance}), we will check the axioms by placing the braids which realize the local moves on the top of the braid diagrams. The following lemma justifies the use of this simplification.

\begin{lemma}
\label{lem:ManipLocal}
Let $L$ be a colored link which coincides with a colored braid $\alpha$ in a small cylinder. Then there exist a colored braid $\beta_r$ (resp. $\beta_l)$ whose closure is isotopic to $L$, and in which $\alpha$ is located at the top right (resp. left) of the braid.
\end{lemma}
 \begin{proof}
Let $L$ be a colored link which coincides with a colored braid $\alpha$ in a small cylinder. Remark~\ref{rem:proofAlexander} ensures the existence of a braid whose closure is $L$, containing $\alpha$ in a small cylinder. First, by conjugation, we bring $\alpha$ to the top of the braid. Then, performing the isotopy depicted in the third diagram of Figure~\ref{fig:braidsmodif}, we move $\alpha$ to the top right (resp. left) of the braid. Finally, as illustrated in the rightmost diagram of Figure~\ref{fig:braidsmodif}, we use conjugation one last time to conclude the proof.
 \end{proof}

\begin{figure}[h] 
\includegraphics[scale=0.6]{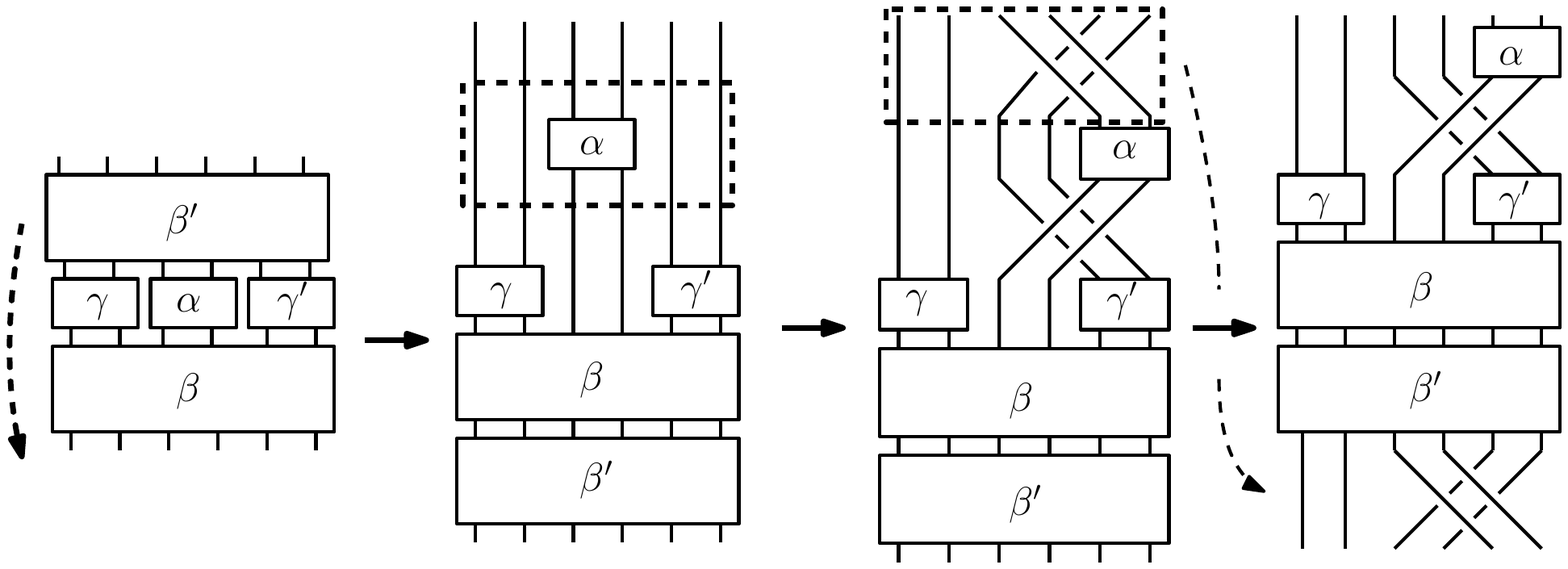}
\caption{Constructing the braid $\beta_r$ by moving $\alpha$ to the upper right.}
\label{fig:braidsmodif}
\end{figure}

We now check that $f_L$ satisfies Jiang's axioms $(R1) \ldots (R5)$. Once the process is completed, we will have concluded the proof of Theorem~\ref{thm:Main}.

 \subsubsection*{Axioms $(R1)$ and $(R2)$}
The fact that $f$ verifies Axiom $(R1)$ was proved in Example \ref{ex:ExampleHopf}. To check that $f$ verifies Axiom $(R2)$, suppose $L$ can be written as the closure of some $n$-stranded $(c,c)$-braid~$\alpha$. Use Lemma~\ref{lem:ManipLocal} to assume that $L \sqcup U$ is obtained as the closure of the $(c',c')$-braid $i_{c_{n+1}}(\alpha)$, where $c'$ is obtained from $c$ by adding an arbitrary additional color $c_{n+1}$. As explained in the proof of Proposition~\ref{prop:Invariance}, the last column of $\overline{\mathcal{B}}_{(c',c')}(i_{c_{n+1}}(\alpha))$ is $(0, \ldots 0, 1)^{T}$. It follows that $\det(\overline{\mathcal{B}}_{(c',c')}(i_{c_{n+1}}(\alpha)) - I_n)$ vanishes and thus so does $f(i_{c_{n+1}}(\alpha))$, as required.

 \subsubsection*{Axiom $(R3)$}
The proof of Axiom $(R3)$ is similar to the proof (given in Proposition~\ref{prop:Invariance}) of the invariance of $f_L$ under the second colored Markov move. Suppose~$L_0$ is obtained as the closure of some $n$-stranded $(c,c)$-braid $\alpha$. We use Lemma~\ref{lem:ManipLocal} to assume that $L'$ is obtained as the closure of $\sigma_n^{-2}i_{c_{n+1}}(\alpha)$; here, $\sigma_n$ is viewed as a $(c',c')$-braid, where $c'$ is obtained from $c$ by adding an arbitrary extra color $c_{n+1}$. The equality we wish to prove is $(t_{c_n} - t_{c_n}^{-1})f(\alpha) = f(\sigma_n^{-2}i_{c_{n+1}}(\alpha))$. Using Definition~\ref{def:Invariantf} and the equality $\langle \sigma_n^{-2}\alpha \rangle=t_{c_n}t_{c_{n+1}} \langle \alpha \rangle$, this reduces to showing the relation
\begin{equation}
\label{eq:Goal2}
\frac{ (t_{c_n} - t_{c_n}^{-1})  g(\det(\Grcc(\alpha)-I_{n-1})) }{t_{c_1} \cdots  t_{c_n} - t^{-1}_{c_1} \cdots  t^{-1}_{c_n}} 
= \frac{-t_{c_n}t_{c_{n+1}} g( \det(\Grcpcp(\sigma_n^{-2}i_{c_{n+1}}(\alpha))-I_{n})) }{t_{c_1} \cdots t_{c_{n+1}} - t^{-1}_{c_1} \cdots t^{-1}_{c_{n+1}}}.
\end{equation}
The aim is now to express the determinant of $\Grcpcp( \sigma_{n}^{-2} i_{c_{n+1}}(\alpha)) - I_{n}$ in terms of the determinant of $\Grcc(\alpha) - I_{n-1}$. As in Proposition~\ref{prop:Invariance}, we write $\overline{\mathcal{B}}_{(c,c)}(\alpha)$ as $\left(\begin{smallmatrix} B & b_1\\ b_2 & a  \end{smallmatrix}\right)$ and $\Grcpcp(i_{c_{n+1}}(\alpha))$ as
\begin{equation}
 \label{matriceexplicitee1}
\Grcpcp(i_{c_{n+1}}(\alpha)) = \begin{pmatrix} 
B & b_1 & 0\\
b_2 & a & 0\\
v_1 & v_2 & 1
\end{pmatrix},
\end{equation}
where $B$ is a square matrix of size $n-2$, $b_1$ is a $(n-2) \times 1$ matrix, $b_2$ and $v_1$ are $1 \times (n-2)$ matrices, and $a$ and $v_2$ belong to~$\Lambda_{\mu}$. Using successively Proposition~\ref{prop:GassnerCocycle} and~(\ref{eq:MatricesReduced}), we deduce that
$$ \Grcpcp(\sigma_n^{-2}i_{c_{n+1}}(\alpha))-I_{n}=\begin{pmatrix} 
B-I_{n-2} & b_1 & (1-t^{-1}_{c_{n+1}})b_1\\
b_2 & a -1 & (1-t^{-1}_{c_{n+1}})a\\
v_1 & v_2 & (1-t^{-1}_{c_{n+1}})v_2 + t^{-1}_{c_{n+1}} t^{-1}_{c_n} - 1
\end{pmatrix}. $$
Just as in the proof of Proposition \ref{prop:Invariance},  our goal is to use Lemma~\ref{lem:NoInfo} and a sequence of elementary operations in order to remove the vectors $v_1$ and $v_2$. Firstly, we subtract to the last column the next-to-last column multiplied by $(1-t^{-1}_{c_{n+1}})$. Secondly, using $A_i$ to denote the rows of the resulting matrix, we multiply the last row of this matrix by $(t_{c_1} \cdots t_{c_n} -~1)$ and add to it $\sum_{i=1}^{n-1}(t_{c_1}\cdots t_{c_i}-1)A_i$. Using Lemma \ref{lem:NoInfo}, we obtain
$$ \det(\Grcpcp(\sigma_n^{-2}i_{c_{n+1}}(\alpha))-I_{n}) = (t_{c_1}  \cdots t_{c_n} - 1)^{-1}\det{
\begin{pmatrix} 
B-I_{n-2} & b_1 & 0\\
b_2 & a -1 & (1-t^{-1}_{c_{n+1}})\\
0 & 0 & e
\end{pmatrix}}, $$
where $e$ is given by $(t^{-1}_{c_{n+1}} t^{-1}_{c_n} - 1)(t_{c_1} \cdots t_{c_n} - 1) + (t_{c_1} \cdots t_{c_{n-1}} - 1)(1-t^{-1}_{c_{n+1}})$.  Finally, computing this latter determinant by expanding along the last row, we deduce that $\det(\Grcpcp(\sigma_n^{-2}i_{c_{n+1}}(\alpha))-~I_{n})$ is equal to 
$$ \frac{(t^{-1}_{c_{n+1}} t^{-1}_{c_n} - 1)(t_{c_1} \cdots t_{c_n} - 1) + (t_{c_1} \cdots t_{c_{n-1}} - 1)(1-t^{-1}_{c_{n+1}})}{(t_{c_1}  \cdots t_{c_n} - 1)}\det({\Grcc(\alpha)-I_{n-1}}).$$
The verification of $(R3)$ is concluded by plugging this result back into~(\ref{eq:Goal2}).

\subsubsection*{Axiom $(R4)$}
Suppose $L_0$ is obtained as the closure of some $n$-stranded $(c,c)$-braid $\alpha$. Using Lemma~\ref{lem:ManipLocal}, we can assume that $L_{--}$ is obtained as the closure of $\sigma_{1}^2\alpha$ and $L_{++}$ as the closure of $\sigma_{1}^{-2}\alpha$; here $\sigma_1$ is viewed as a $((c_1, c_2, c_3,\ldots, c_n),(c_2, c_1, c_3,\ldots, c_n))$-braid.
The relation we wish to prove is $f(L_{--})+f(L_{++})=f(L_0)$. Using Definition~\ref{def:Invariantf} and performing some simplifications, this reduces to
\begin{equation}\label{eq:axiome2}
\frac{g(\det(\Grcc(\sigma_{1}^2\alpha)-I_{n-1}))}{t_{c_1} t_{c_2}} + \frac{g(\det(\Grcc(\sigma_{1}^{-2}\alpha)-I_{n-1}))}{t_{c_1}^{-1} t_{c_2}^{-2}} = \\ (t_{c_1} t_{c_2} + t_{c_1}^{-1} t_{c_2}^{-1})g(\det(\Grcc(\alpha)-I_{n-1})).
\end{equation}
In order to check~(\ref{eq:axiome2}), we must compute $g(\det(\Grcc(\sigma_{1}^{\pm 2}\alpha)-I_{n-1}))$. To this end, we write $\Grcc(\alpha)=\left(\begin{smallmatrix} 
a & c & p \\
b & d & q\\
x & y & M \end{smallmatrix} \right),$
where $a$,$b$,$c$ and $d$ are elements of $\Lambda_{\mu}$, $p$ and $q$ are rows of length $(n-3)$, $x$ and $y$ are columns of length $(n-3)$, and $M$ is a square matrix of size $(n-3)$. Using successively~(\ref{eq:MatricesReduced}) and Proposition~\ref{prop:GassnerCocycle}, we deduce that
$$ \Grcc(\sigma_{1}^2\alpha)-I_{n-1}=\begin{pmatrix} 
t_{c_1} t_{c_2} a + (1-t_{c_1})c - 1& c & p\\
 t_{c_1} t_{c_2} b + (1-t_{c_1})d & d-1 & q\\
 t_{c_1} t_{c_2}x + (1-t_{c_1})y & y & M - I_{n-3}
\end{pmatrix} $$
and we use $A^+$ to denote the first column of this matrix. A similar computation yields
$$\Grcc(\sigma_{1}^{-2}\alpha)-I_{n-1}= \begin{pmatrix} 
t_{c_1}^{-1} t_{c_2}^{-1} a + (t_{c_2}^{-1} - t_{c_2}^{-1} t_{c_1}^{-1})c - 1& c & p \\
t_{c_1}^{-1} t_{c_2}^{-1} b + (t_{c_2}^{-1} - t_{c_2}^{-1} t_{c_1}^{-1})d  & d-1 & q\\
t_{c_1}^{-1} t_{c_2}^{-1} x + (t_{c_2}^{-1} - t_{c_2}^{-1} t_{c_1}^{-1})y & y & M - I_{n-3}
\end{pmatrix} $$
and we use $A^-$ (resp. $A^0$) to denote the first column of this latter matrix (resp. $\Grcc(\alpha)-I_{n-1}$). Furthermore, a direct computation shows that the following relation holds: 
\begin{equation}
\label{eq:FirstColumn}
\frac{1}{t_{c_1} t_{c_2}}g(A^{+}) + \frac{1}{t_{c_1}^{-1} t_{c_2}^{-1}}g(A^{-}) = (t_{c_1} t_{c_2} + t_{c_1}^{-1} t_{c_2}^{-1})g(A^{0}).
\end{equation}
We can now check~(\ref{eq:axiome2}). Indeed, as the three matrices involved in~(\ref{eq:axiome2}) differ only in their first column, this relation follows by expanding the determinants  with respect to their first column and applying~(\ref{eq:FirstColumn}). This concludes the verification of Axiom $(R4)$.

\subsubsection*{Axiom $(R5)$}

Using Lemma~\ref{lem:ManipLocal}, assume that $ L(1),\ldots,L(6)$ are respectively obtained as the closures of $\beta_1\alpha,\ldots, \beta_6 \alpha$ for some $((c_3, c_2, c_1, c_4, \ldots ,c_n),(c_1, c_2, c_3, c_4, \ldots ,c_n))$-braid $\alpha$, and where $\beta_1,\ldots,\beta_6$ are the $((c_1, c_2, c_3, c_4, \ldots ,c_n),(c_3, c_2, c_1, c_4, \ldots ,c_n))$-braids depicted in Figure~\ref{fig:JiangIIIBraids}.
\begin{figure}[!htb]
\includegraphics[scale=0.8]{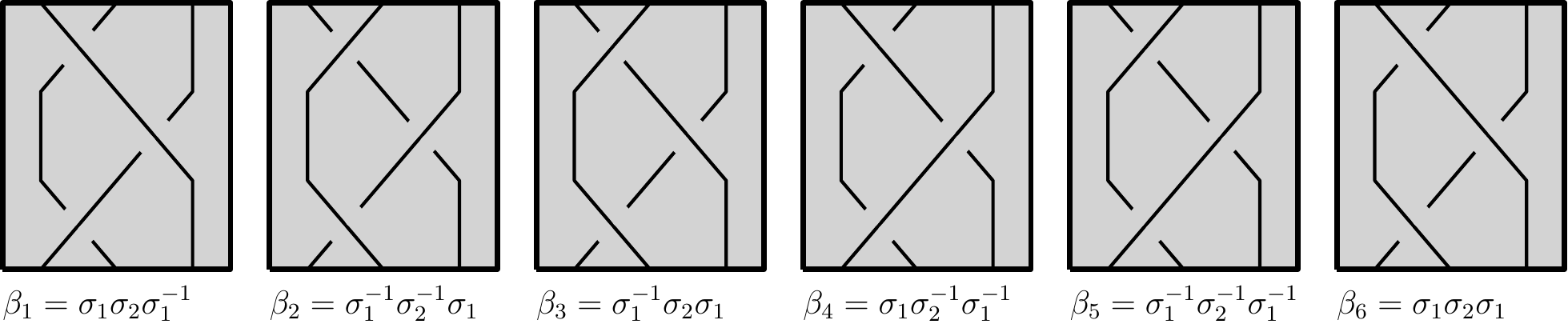}
\caption{The braids $\beta_1,\ldots,\beta_6$ involved in the verification of axiom $(R5)$.}
\label{fig:JiangIIIBraids}
\end{figure}

As usual, we start by rewriting the axiom in a more convenient fashion. Namely, after using Definition~\ref{def:Invariantf} and simplifying the signs and the $\langle \alpha \rangle$'s, the axiom reduces to verifying the following equation:
\begin{multline}\label{eq:axiomIII}
(t^{-1}_{c_1} t^{-1}_{c_2} - t_{c_1} t_{c_2})\left[\frac{1}{t^2_{c_1} t^{-1}_{c_3}} g(\det(\Grcc(\beta_1\alpha)-I_{n-1})) + \frac{1}{t^{-1}_{c_3}} g(\det(\Grcc(\beta_2\alpha)-I_{n-1}))\right] \\
+ (t_{c_2} t_{c_3} - t^{-1}_{c_2} t^{-1}_{c_3})\left[\frac{1}{t_{c_1}} g(\det(\Grcc(\beta_3\alpha)-I_{n-1})) + \frac{1}{t_{c_1} t^{-2}_{c_3}} g(\det(\Grcc(\beta_4\alpha)-I_{n-1}))\right] \\
+ (t_{c_1} t_{c_3} - t^{-1}_{c_1} t^{-1}_{c_3})\left[\frac{1}{t^{-1}_{c_2} t^{-2}_{c_3}} g(\det(\Grcc(\beta_5\alpha)-I_{n-1})) + \frac{1}{t^2_{c_1} t_{c_2}} g(\det(\Grcc(\beta_6\alpha)-I_{n-1}))\right] = 0.
\end{multline}

Since our aim is to compute each of the $g(\det(\Grcc(\beta_i\alpha)-I_{n-1})$, we start by writing out~$\Grcc(\alpha)$ as the matrix $\left(\begin{smallmatrix}\mathbf{a}\\ \mathbf{b}\\ \mathbf{e}\\ \mathbf{D} \end{smallmatrix}\right)^{T}$ where $\mathbf{a}$, $\mathbf{b}$ and $\mathbf{e}$ are rows of length $(n-1)$, and $\mathbf{D}$ is a matrix of size $(n-4) \times (n-1)$. Using successively~(\ref{eq:MatricesReduced}) and Proposition~\ref{prop:GassnerCocycle}, we deduce that the reduced colored Gassner matrices $\overline{\mathcal{B}}_{(c,c)}(\beta_1 \alpha),\ldots,\overline{\mathcal{B}}_{(c,c)}(\beta_6 \alpha)$ are respectively given by
$$ \begin{bmatrix}  -t_{c_1}\mathbf{b} + \mathbf{e}\\
-t_{c_1} t^{-1}_{c_3} \mathbf{a} + (t_{c_1} t^{-1}_{c_3} - t_{c_1})\mathbf{b} + \mathbf{e}\\
\mathbf{e}\\
\mathbf{D} \end{bmatrix}^{T} \qquad
\begin{bmatrix}  -t^{-1}_{c_2} t^{-1}_{c_3}\mathbf{b} + t^{-1}_{c_2} t^{-1}_{c_3}\mathbf{e}\\
-t_{c_2} \mathbf{a} + (1 - t^{-1}_{c_3})\mathbf{b} + t^{-1}_{c_3}\mathbf{e}\\
\mathbf{e}\\
\mathbf{D} \end{bmatrix}^{T} \qquad
\begin{bmatrix}  (1-t_{c_1})\mathbf{a} - t^{-1}_{c_2}\mathbf{b} + t^{-1}_{c_2}\mathbf{e} \\
-t_{c_1} t_{c_2} \mathbf{a} + \mathbf{e}\\
\mathbf{e}\\
\mathbf{D} \end{bmatrix}^{T} $$

$$ \begin{bmatrix}  - t^{-1}_{c_3} (1 - t_{c_1})\mathbf{a} - t_{c_1} t^{-1}_{c_3}\mathbf{b} + t^{-1}_{c_3}\mathbf{e}\\
-t^{-1}_{c_3} \mathbf{a} + t^{-1}_{c_3}\mathbf{e}\\
\mathbf{e}\\
\mathbf{D} \end{bmatrix}^{T}  \qquad
\begin{bmatrix}  -t^{-1}_{c_2} t^{-1}_{c_3}\mathbf{b} + t^{-1}_{c_2} t^{-1}_{c_3}\mathbf{e}\\
-t^{-1}_{c_3} \mathbf{a} + t^{-1}_{c_3}\mathbf{e}\\
\mathbf{e}\\
\mathbf{D} \end{bmatrix}^{T} \qquad
\begin{bmatrix}  -t_{c_1}\mathbf{b} + \mathbf{e} \\
-t_{c_1} t_{c_2} \mathbf{a} + \mathbf{e}\\
\mathbf{e}\\
\mathbf{D}  \end{bmatrix}^{T}. $$
Our first goal is to get rid of the $\textbf{e}$ in the first and second columns of $\overline{\mathcal{B}}_{(c,c)}(\beta_1 \alpha)-I_{n-1},\ldots,\overline{\mathcal{B}}_{(c,c)}(\beta_6 \alpha)-I_{n-1}$. This is done by subtracting the appropriate multiple of the third column from the first and second columns (notice that this operation does not change the determinant). We denote the resulting matrices by $M^1,\ldots,M^6$. As an illustration, we perform this operation on
$$\Grcc(\beta_6\alpha) - I_{n-1}= \begin{bmatrix} 
-t_{c_1} b_1 + e_1 - 1 & -t_{c_1} t_{c_2} a_1 + e_1 &  e_1  & \mathbf{d_1} \\
 -t_{c_1} b_2 + e_2 & -t_{c_1} t_{c_2} a_2 + e_2 - 1 & e_2 & \mathbf{d_2}\\
-t_{c_1} b_3 + e_3& -t_{c_1} t_{c_2} a_3 + e_3 & e_3 - 1 & \mathbf{d_3}\\
\ldots & \ldots & \ldots & \\
-t_{c_1} b_{n-1} + e_{n-1}&  -t_{c_1} t_{c_2} a_{n-1} + e_{n-1} &e_{n-1} & \mathbf{D'} - I_{n-4}
\end{bmatrix}, $$
where $\mathbf{d_1}$, $\mathbf{d_2}$, and $\mathbf{d_3}$ are the first three rows of $\mathbf{D}^T$, and $\mathbf{D'}$ is the $(n-4) \times (n-4)$-matrix made of the remaining rows of $\mathbf{D}^T$. Subtracting the third column from the first and second columns, we get:
$$M^6=\begin{bmatrix} 
-t_{c_1} b_1 - 1 & -t_{c_1} t_{c_2} a_1 &  e_1  & \mathbf{d_1} \\
-t_{c_1} b_2 & -t_{c_1} t_{c_2} a_2 - 1 & e_2 & \mathbf{d_2}\\
-t_{c_1} b_3 & -t_{c_1} t_{c_2} a_3 & e_3 - 1 & \mathbf{d_3}\\
\ldots & \ldots & \ldots & \\
-t_{c_1} b_{n-1} &  -t_{c_1} t_{c_2} a_{n-1} &e_{n-1} & \mathbf{D'} - I_{n-4}
\end{bmatrix}. $$
In order to conclude the verification of $(R5)$, the idea is now to consider a subset $M^l_{i,j}$ of the collection of all $2 \times 2 $ minors of the matrices $M^1,\ldots, M^6$ and to show~(\ref{eq:axiomIII}) for the $M^l_{i,j}$. In more details, for $0 \leq i < j \leq n-1$, and $l \in \{1, 2, \ldots 6 \}$, we use $M^{l}_{i,j}$ to denote the determinant of the $2 \times 2$ matrix obtained from $M^{l}$ by removing all columns but the first two, and all rows except the $i^{th}$ and~$j^{th}$. As we shall argue below, the following claim implies~(\ref{eq:axiomIII}):
\begin{claim} \label{claim}
For each $i$ and $j$ as above, we have the following equality:
	\begin{multline*}
	(t^{-1}_{c_1} t^{-1}_{c_2} - t_{c_1} t_{c_2})\left[\frac{1}{t^2_{c_1} t^{-1}_{c_3}} g(M^1_{i,j}) + \frac{1}{t^{-1}_{c_3}} g(M^2_{i,j})\right] + 
	(t_{c_2} t_{c_3} - t^{-1}_{c_2} t^{-1}_{c_3})\left[\frac{1}{t_{c_1}} g(M^3_{i,j}) + \frac{1}{t_{c_1} t^{-2}_{c_3}} g(M^4_{i,j})\right] \\
	+ (t_{c_1} t_{c_3} - t^{-1}_{c_1} t^{-1}_{c_3})\left[\frac{1}{t^{-1}_{c_2} t^{-2}_{c_3}} g(M^5_{i,j}) + \frac{1}{t^2_{c_1} t_{c_2}} g(M^6_{i,j})\right] = 0.
	\end{multline*}
\end{claim}

The proof of this claim is a tedious but direct calculation since (despite the high number of minors) it actually only involves 7 distinct types of computations. Indeed, for $i,j \geq 4$, all the~$M_{i,j}$ are computed from matrices of the same form (albeit with different indices). We refer to~\cite[Appendix]{EstierMaster} for examples of these computations.

It remains to argue why the claim concludes the verification of axiom $(R5)$. As we explained above, the axiom will follow once we show that~(\ref{eq:axiomIII}) holds with each $\Grcc(\beta_l\alpha) - I_{n-1}$ replaced by the corresponding $M^l$. To obtain this latter equality, we successively expand each determinant along its columns (starting from the rightmost column and progressing to the left) until there remain six sums of the aforementioned $2 \times 2$ minors. The assertion then follows by grouping up the determinants according to their coefficients, and using the claim. This concludes the proof of Theorem~\ref{thm:Main}.

\section{Homological interpretation of the reduced colored Gassner representation}
\label{sec:Homology}

The aim of this section is to prove Theorem~\ref{thm:Second} which provides an intrinsic definition of the reduced colored Gassner matrices and relates them to the so-called \emph{reduced colored Gassner representation}~\cite{KirkLivingston, CimasoniConwayGG}. To achieve this, Subsection~\ref{sub:FreelyGenerate} starts by providing a homological interpretation of the elements $\widetilde{g}_1,\ldots,\widetilde{g}_{n-1}$, while Subsection~\ref{sub:Agree} concludes the proof of Theorem~\ref{thm:Second}.

\subsection{Preliminary lemmas}
\label{sub:FreelyGenerate}

Fix a sequence $c=(c_1,\ldots,c_n)$ of integers in $\lbrace 1,\ldots, \mu \rbrace$ and a basepoint $z$ for $D_c$ which lies in its (unique) boundary component $\partial D_c$. Recall that $p \colon \widehat{D}_c \to D_c$ denotes the regular cover corresponding to the kernel of $\psi_c \colon \pi_1(D_c) \to \Z^\mu, x_i \mapsto t_{c_i}$. We still write $P$ for the fiber $p^{-1}(z)$ over $z$ and we use the notation $\partial \widehat{D}_c \to \partial D_c$ for the restriction of $p$ to $\partial D_c$. Finally, recall from Section~\ref{sec:Prelim} that $\pi_1(D_n,z)$ is freely generated either by the loops $x_1,\ldots,x_n$ depicted in Figure~\ref{fig:DiskGeneratorsThesis} or by $g_1,\ldots, g_n$, where $g_i=x_1\cdots x_i$. From now on, we will assume that~$g_n$ lies in $\partial D_c$.

In order to provide a homological interpretation of the $\widetilde{g}_i$, we start with a preliminary lemma.

\begin{lemma}
\label{lem:ShortExact}
The long exact sequence of the triple $(\widehat{D}_c,\partial \widehat{D}_c,P)$ gives rise to the short exact sequence
$$ 0 \to H_1(\partial \widehat{D}_c,P) \stackrel{j}{\to} H_1(\widehat{D}_c,P) \stackrel{\pi}{\to} H_1(\widehat{D}_c,\partial \widehat{D}_c) \to 0.$$
Furthermore, $\im(j)$ is freely generated by $\widetilde{g}_n$.
\end{lemma}
\begin{proof}
To prove both claims, we must understand the $\Lambda_\mu$-module $H_i(\partial \widehat{D}_c,P)$ for $i=0,1$. Since the covering $\partial \widehat{D}_c \to \partial D_c$ arises from the restriction of the homomorphism $\psi_c \colon \pi_1(D_c) \to~\Z^\mu$ to $\pi_1(\partial D_c)$, it consists in a disjoint union of copies of the regular cover $\R \to \partial D_c$ with deck transformation generator $t_{c_1}\cdots t_{c_n}$. It follows that $H_0(\partial \widehat{D}_c,P)$ vanishes (give the circle $\partial D_c$ its usual cell structure with $z$ as its unique $0$-cell; it follows that the $0$-skeleton of $\widehat{D}_c$ is given by $P$). The first assertion is now immediate since $H_2(\widehat{D}_c,\partial \widehat{D}_c)$ also vanishes. The second assertion follows from similar topological considerations.
\end{proof}

Just as in Lemma~\ref{lem:ShortExact}, we use $\pi$ to denote the inclusion induced map $H_1(\widehat{D}_c,P) \to H_1(\widehat{D}_c,\partial \widehat{D}_c)$. 

\begin{lemma}
\label{lem:FreelyGenerate}
The $\Lambda_\mu$-module $H_1(\widehat{D}_c,\partial \widehat{D}_c)$ is freely generated by $\pi(\widetilde{g}_1),\ldots,\pi(\widetilde{g}_{n-1})$.
\end{lemma}
\begin{proof}
In order to show that $\pi(\widetilde{g}_1),\ldots,\pi(\widetilde{g}_{n-1})$ are linearly independent, assume that the linear combination~$\sum_{i=0}^{n-1} \lambda_i \pi(\widetilde{g}_{i})$ vanishes for some $\lambda_i$ in $\Lambda_\mu$. By exactness of the sequence displayed in Lemma~\ref{lem:ShortExact}, there is an $x$ in $H_1(\partial \widehat{D}_c,P)$ such that~$j(x)=\sum_{j=1}^{n-1}\lambda_i \widetilde{g}_i$. Since Lemma~\ref{lem:ShortExact} implies that $\im(j)$ is freely generated by $\widetilde{g}_n$, we deduce that there is a $\lambda \in \Lambda_\mu$ for which $j(x)=\lambda \widetilde{g}_n$. The result now follows from the fact that $\widetilde{g}_1,\ldots,\widetilde{g}_n$ form a basis of~$H_1(\widehat{D}_c,P)$.

Next, we show that $\pi(\widetilde{g}_1),\ldots,\pi(\widetilde{g}_{n-1})$ generate $H_1(\widehat{D}_c,\partial \widehat{D}_c)$. Given $x \in H_1(\widehat{D}_c,\partial \widehat{D}_c)$, we can find some $\lambda_1,\ldots,\lambda_n$ in $\Lambda_\mu$ such that $x=\pi(\sum_{i=1}^n \lambda_i \widetilde{g}_i)$: indeed $\pi$ is surjective thanks to Lemma~\ref{lem:ShortExact} and the $\widetilde{g}_1,\ldots,\widetilde{g}_n$ form a basis of $H_1(\widehat{D}_c,P)$. To prove the assertion, we must show that $\pi(\widetilde{g}_n)$ vanishes, but this is immediate since $\widetilde{g}_n$ lies in $\partial \widehat{D}_c$.
\end{proof}

\subsection{Relation to the reduced colored Gassner representation}
\label{sub:Agree}
Let $S$ be the multiplicative subset of $\Lambda_\mu$ generated by $(1-t_1), \ldots, (1-t_\mu)$ and let $\Lambda_S$ be the localization of~$\Lambda_\mu$ with respect to $S$. Fix a self-homeomorphism $h_\beta$ representing a $(c,c)$-braid $\beta$. Lifting $h_\beta$ to~$\widehat{D}_c$ gives rise to a well-defined automorphism $(\widetilde{h}_\beta)_*$ of $\Lambda_S \otimes_ {\Lambda_\mu} H_1(\widehat{D}_c)$. The \emph{reduced colored Gassner representation}
$$ B_c \to \operatorname{Aut}_{\Lambda_S}(\Lambda_S \otimes_ {\Lambda_\mu} H_1(\widehat{D}_c))$$
is obtained by mapping a braid $\beta$ to $(\widetilde{h}_\beta)_*$. Kirk-Livingston-Wang~\cite{KirkLivingstonWang} initially defined this representation using coefficients in $Q$, the field of fractions of $\Lambda_\mu$. To the best of our knowledge, the first use of $\Lambda_S$-coefficients in this setting occured in~\cite{CimasoniConwayGG}, see also~\cite[Section 9.4]{ConwayThesis}. Note that these localizations are performed because $H_1(\widehat{D}_c)$ is not free for $\mu >2$ while the $\Lambda_S$-module $\Lambda_S \otimes_{\Lambda_\mu} H_1(\widehat{D}_c)$ is always free~\cite[Lemma 9.4.6]{ConwayThesis}.

Finally, given a $(c,c)$-braid $\beta$, recall that a homeomorphism~$h_\beta$ representing $\beta$ induces a map on $\Lambda_S \otimes_ {\Lambda_\mu} H_1(\widehat{D}_c,\partial \widehat{D}_c)$. We are ready to prove Theorem~\ref{thm:Second} whose statement we recall for the reader's convenience.

\begin{customthm}{\ref{thm:Second}}
\label{thm:SecondEnd}
Given a $(c,c)$-braid $\beta$, the following statements hold:
\begin{enumerate}
\item The map induced by $\beta$ on $\Lambda_S \otimes_ {\Lambda_\mu} H_1(\widehat{D}_c,\partial \widehat{D}_c)$ is represented by the reduced colored Gassner matrix $\overline{\mathcal{B}}_{(c,c)}(\beta)$.
\item The inclusion induced homomorphism $\Phi \colon \Lambda_S \otimes_ {\Lambda_\mu} H_1(\widehat{D}_c) \to \Lambda_S \otimes_ {\Lambda_\mu} H_1(\widehat{D}_c,\partial \widehat{D}_c)$ intertwines the reduced colored Gassner representation with the map induced by $\beta$. Furthermore, after tensoring with $Q$, the induced map $\operatorname{id}_Q \otimes \Phi$ is an isomorphism which conjugates the two representations.
\end{enumerate}
\end{customthm}

\begin{proof}
To prove the first assertion, recall that by definition, the reduced colored Gassner matrix is the restriction of the unreduced colored Gassner representation to the free submodule of $H_1(\widehat{D}_c,P)$ generated by the $\widetilde{g}_1,\ldots, \widetilde{g}_{n-1}$. Since the unreduced colored Gassner representation is the automorphism of $H_1(\widehat{D}_c,P)$ induced by $\beta$, the result now immediately follows from Lemma~\ref{lem:FreelyGenerate}. To prove the second assertion, consider the long exact sequence of the pair $(\widehat{D}_c,\partial \widehat{D}_c)$. Tensoring with $\Lambda_S$, which is flat over $\Lambda_\mu$, we obtain the exact sequence
$$ 0 \to \Lambda_S \otimes_ {\Lambda_\mu} H_1(\widehat{D}_c) \to \Lambda_S \otimes_ {\Lambda_\mu} H_1(\widehat{D}_c,\partial \widehat{D}_c) \to \Lambda_S \otimes_ {\Lambda_\mu} H_0(\partial \widehat{D}_c). $$
Since both representations are induced by $\widetilde{h}_\beta$, the naturality of the long exact sequence in homology implies that the homomorphism $\Phi$ induced by the inclusion map $(\widehat{D}_c,\emptyset) \to (\widehat{D}_c,\partial \widehat{D}_c)$ satisfies the required property. Since $H_0(\partial \widehat{D}_c) \cong \Lambda_\mu /(t_{c_1}\cdots t_{c_n}-1)$, passing to~$Q$ coefficients, $Q \otimes_ {\Lambda_\mu} H_0(\partial \widehat{D}_c)$ vanishes  and the final assertion follows.
\end{proof}

\appendix
\section{A second proof of Theorem~\ref{thm:Main}.}
\label{Appendix}

This appendix contains an alternative proof of Theorem~\ref{thm:Main} that was suggested to us by a kind referee. This proof relies on articles of Morton~\cite{MortonBraids} and Hartley~\cite{Hartley} but has two notable advantages: firstly it is much shorter than the one given in Section~\ref{sec:Thm} and secondly it is more geometrical in nature.

\begin{proof}[Alternative proof of Theorem~\ref{thm:Main}]
We work in the case $\mu=n$ for simplicity.
Use $A$ to denote the simple closed curve $\partial D_n$, oriented with the clockwise orientation. 
View $A \cup \widehat{\beta}$ as an $(n+1)$-colored link, and use $x$ to denote the variable of  $\Delta_{\widehat{\beta} \cup A}$ corresponding to the component $A$.
A theorem due to Morton~\cite[Theorem 1]{MortonBraids} relates $\Delta_{\widehat{\beta} \cup A}$ to the colored Gassner representation of $\beta$.
Using our conventions, this result reads as
$$ \Delta_{\widehat{\beta} \cup A}(t_1^2,\ldots,t_n^2,x^2)=g(\det(x^{-1} \overline{\mathcal{B}}_{(c,c)}(\beta)-I_{n-1})).$$
Consequently, we deduce that $x^{n-1} \langle \beta \rangle g(\det(x^{-1} \overline{\mathcal{B}}_{(c,c)}(\beta)-I_{n-1}))$ is symmetric up to a sign. We therefore obtain the following equation for a certain $\varepsilon$ that remains to be determined:
\begin{equation}
\label{eq:MortonPotential}
 \nabla_{\widehat{\beta} \cup A}(t_1,\ldots,t_n,x)=\varepsilon x^{n-1} \langle \beta \rangle g (\det(x^{-1} \overline{\mathcal{B}}_{(c,c)}(\beta)-I_{n-1})).
\end{equation}

We claim that $(-1)^{n-1}\varepsilon =1$. To achieve this, we compute the highest degree monomial of~$\nabla_{\widehat{\beta} \cup A}(1,\ldots,1,x)$ in two different ways. On the one hand, if we set $t_i=1$ in the right hand side of~\eqref{eq:MortonPotential}, then the highest degree monomial in the resulting expression is $ \varepsilon (-1)^{n-1} x^{n-1}$. On the other hand, if we use $\lambda_i$ to denote the linking number of the $i$-th component of $\widehat{\beta}$ with the axis $A$, then an application of~\cite[Equation~5.4]{Hartley} yields
\begin{equation}
\label{eq:UseHartley}
\nabla_{\widehat{\beta} \cup A}(1,\ldots,1,x)=\nabla_A(x)\prod_{i=1}^n (x^{\lambda_i}-x^{-\lambda_i}).
\end{equation}
Since $A$ is an unknot, we have $\nabla_A(x)=(x-x^{-1})^{-1}$, and since all the linking numbers $\lambda_i$ are positive, we know that $\sum_{i=1}^n \lambda_i=n$. We therefore deduce that the highest degree monomial in~\eqref{eq:UseHartley} is $x^{n-1}$. This proves  the claim.

We now conclude the proof of the theorem by deducing the potential function~$\nabla_{\widehat{\beta}}$ from~$\nabla_{\widehat{\beta} \cup A}$. To that end, we set $x=1$ in~\eqref{eq:MortonPotential}, use the claim and apply Hartley's normalisation of the Torres formula~\cite[Equation~5.3]{Hartley} to obtain 
\begin{align*}
 \nabla_{\widehat{\beta}}(t_1,\ldots,t_n)
 &=\frac{1}{(t_1\cdots t_n-t_1^{-1}\cdots t_n^{-1})}\nabla_{\widehat{\beta} \cup A}(t_1,\ldots,t_n,1) \\
 &=\frac{1}{(t_1\cdots t_{n}-t_1^{-1}\cdots t_{n}^{-1})} (-1)^{n-1} \langle \beta \rangle g (\det( \overline{\mathcal{B}}_{(c,c)}(\beta)-I_{n-1})
 \end{align*}
 This concludes the alternative proof of the theorem.
\end{proof}

\bibliography{BiblioBurauPotential}

\begin{thebibliography}{10}

\bibitem{AlexanderClosure}
James Alexander.
\newblock A lemma on a system of knotted curves.
\newblock {\em Proc. Natl. Acad. Sci. USA.}, 9:93--95, 1923.

\bibitem{BenAribiConway}
Fathi Ben~Aribi and Anthony Conway.
\newblock {$L^2$}-{B}urau maps and {$L^2$}-{A}lexander torsions.
\newblock {\em Osaka J. Math.}, 55(3):529--545, 2018.

\bibitem{BenheddiCimasoni}
Mounir Benheddi and David Cimasoni.
\newblock Link {F}loer homology categorifies the {C}onway function.
\newblock {\em Proc. Edinb. Math. Soc. (2)}, 59(4):813--836, 2016.

\bibitem{Birman}
Joan~S. Birman.
\newblock {\em Braids, links, and mapping class groups}.
\newblock Princeton University Press, Princeton, N.J.; University of Tokyo
  Press, Tokyo, 1974.
\newblock Annals of Mathematics Studies, No. 82.

\bibitem{BirmanBrendle}
Joan~S. Birman and Tara~E. Brendle.
\newblock Braids: a survey.
\newblock In {\em Handbook of knot theory}, pages 19--103. Elsevier B. V.,
  Amsterdam, 2005.

\bibitem{Burau}
Werner Burau.
\newblock \"{U}ber {Z}opfgruppen und gleichsinnig verdrillte {V}erkettungen.
\newblock {\em Abh. Math. Sem. Univ. Hamburg}, 11(1):179--186, 1935.

\bibitem{CimasoniPotential}
David Cimasoni.
\newblock A geometric construction of the {C}onway potential function.
\newblock {\em Comment. Math. Helv.}, 79(1):124--146, 2004.

\bibitem{CimasoniConway2Functor}
David Cimasoni and Anthony Conway.
\newblock A {B}urau-{A}lexander 2-functor on tangles.
\newblock {\em Fund. Math.}, 240(1):51--79, 2018.

\bibitem{CimasoniConwayGG}
David Cimasoni and Anthony Conway.
\newblock Coloured tangles and signatures.
\newblock {\em Math. Proc. Cambridge Philos. Soc.}, 164(3):493--530, 2018.

\bibitem{CimasoniTuraev}
David Cimasoni and Vladimir Turaev.
\newblock A {L}agrangian representation of tangles.
\newblock {\em Topology}, 44(4):747--767, 2005.

\bibitem{ConwayThesis}
Anthony Conway.
\newblock Invariants of colored links and generalizations of the {B}urau
  representation.
\newblock 2017.
\newblock {U}niversity of {G}eneva.

\bibitem{ConwayTwistedBurau}
Anthony Conway.
\newblock Burau maps and twisted {A}lexander polynomials.
\newblock {\em Proc. Edinb. Math. Soc. (2)}, 61(2):479--497, 2018.

\bibitem{ConwayJohn}
John Conway.
\newblock An enumeration of knots and links, and some of their algebraic
  properties.
\newblock In {\em Computational {P}roblems in {A}bstract {A}lgebra ({P}roc.
  {C}onf., {O}xford, 1967)}, pages 329--358. Pergamon, Oxford, 1970.

\bibitem{DeguchiAkutsu}
Tetsuo Deguchi and Yasuhiro Akutsu.
\newblock Graded solutions of the {Y}ang-{B}axter relation and link
  polynomials.
\newblock {\em J. Phys. A}, 23(11):1861--1875, 1990.

\bibitem{EstierMaster}
Solenn Estier.
\newblock Colored {G}assner matrices and {C}onway’s potential function.
\newblock 2017.
\newblock Master thesis, {U}niversity of {G}eneva.

\bibitem{FoxFree2}
Ralph~H. Fox.
\newblock Free differential calculus. {II}. {T}he isomorphism problem of
  groups.
\newblock {\em Ann. of Math. (2)}, 59:196--210, 1954.

\bibitem{Hartley}
Richard Hartley.
\newblock The {C}onway potential function for links.
\newblock {\em Comment. Math. Helv.}, 58(3):365--378, 1983.

\bibitem{Jiang}
Bo~Ju Jiang.
\newblock On {C}onway's potential function for colored links.
\newblock {\em Acta Math. Sin. (Engl. Ser.)}, 32(1):25--39, 2016.

\bibitem{TuraevKassel}
Christian Kassel and Vladimir Turaev.
\newblock {\em Braid groups}, volume 247 of {\em Graduate Texts in
  Mathematics}.
\newblock Springer, New York, 2008.
\newblock With the graphical assistance of Olivier Dodane.

\bibitem{KauffmanSaleur}
L.~H. Kauffman and H.~Saleur.
\newblock Free fermions and the {A}lexander-{C}onway polynomial.
\newblock {\em Comm. Math. Phys.}, 141(2):293--327, 1991.

\bibitem{KauffmanConway}
Louis~H. Kauffman.
\newblock The {C}onway polynomial.
\newblock {\em Topology}, 20(1):101--108, 1981.

\bibitem{KirkLivingston}
Paul Kirk and Charles Livingston.
\newblock Twisted {A}lexander invariants, {R}eidemeister torsion, and
  {C}asson-{G}ordon invariants.
\newblock {\em Topology}, 38(3):635--661, 1999.

\bibitem{KirkLivingstonWang}
Paul Kirk, Charles Livingston, and Zhenghan Wang.
\newblock The {G}assner representation for string links.
\newblock {\em Commun. Contemp. Math.}, 3(1):87--136, 2001.

\bibitem{Markov}
Andrey Markov.
\newblock {U}ber die freie {A}quivalenz geschlossener zopfe.
\newblock {\em Recueil Math. Moscou}, 1:73--78, 1935.

\bibitem{MilnorDuality}
John Milnor.
\newblock A duality theorem for {R}eidemeister torsion.
\newblock {\em Ann. of Math. (2)}, 76:137--147, 1962.

\bibitem{MortonBraids}
Hugh Morton.
\newblock The multivariable {A}lexander polynomial for a closed braid.
\newblock In {\em Low-dimensional topology ({F}unchal, 1998)}, volume 233 of
  {\em Contemp. Math.}, pages 167--172. Amer. Math. Soc., Providence, RI, 1999.

\bibitem{MortonHodgson}
Hugh Morton and Julian Hodgson.
\newblock multiburau.
\newblock https://livrepository.liverpool.ac.uk/2048779/.

\bibitem{Murakami}
Hitoshi Murakami.
\newblock A weight system derived from the multivariable {C}onway potential
  function.
\newblock {\em J. London Math. Soc. (2)}, 59(2):698--714, 1999.

\bibitem{OzsvathSzaboHolomorphic}
Peter Ozsv\'ath and Zolt\'an Szab\'o.
\newblock Holomorphic disks and knot invariants.
\newblock {\em Adv. Math.}, 186(1):58--116, 2004.

\bibitem{OzsvathSzabo}
Peter Ozsv\'ath and Zolt\'an Szab\'o.
\newblock Holomorphic disks, link invariants and the multi-variable {A}lexander
  polynomial.
\newblock {\em Algebr. Geom. Topol.}, 8(2):615--692, 2008.

\bibitem{Penne1}
Rudi Penne.
\newblock Multi-variable {B}urau matrices and labeled line configurations.
\newblock {\em J. Knot Theory Ramifications}, 4(2):235--262, 1995.

\bibitem{Penne2}
Rudi Penne.
\newblock The {A}lexander polynomial of a configuration of skew lines in
  {$3$}-space.
\newblock {\em Pacific J. Math.}, 186(2):315--348, 1998.

\bibitem{Seifert}
Herbert Seifert.
\newblock \"{U}ber das {G}eschlecht von {K}noten.
\newblock {\em Math. Ann.}, 110(1):571--592, 1935.

\bibitem{TuraevReidemeister}
Vladimir Turaev.
\newblock Reidemeister torsion in knot theory.
\newblock {\em Uspekhi Mat. Nauk}, 41(1(247)):97--147, 240, 1986.

\bibitem{TuraevFaithful}
Vladimir Turaev.
\newblock Faithful linear representations of the braid groups.
\newblock {\em Ast\'erisque}, (276):389--409, 2002.
\newblock S{\'e}minaire Bourbaki, Vol. 1999/2000.

\end{thebibliography}
\bibliographystyle{plain}

\end{document}